\documentclass[10pt,a4paper,oneside]{article}
\usepackage[utf8]{inputenc}
\usepackage[english]{babel}

\pdfoutput=1

\usepackage{amsmath,amsfonts,amssymb,amsopn,amscd}
\usepackage{xy}
\usepackage{graphicx}

\usepackage{datetime}
\newdateformat{generalisodate}{\THEYEAR.\twodigit\THEMONTH.\twodigit\THEDAY}
\newcommand{\todayiso}{\the\year\twodigit\month\twodigit\day}

\usepackage{fancyhdr}
\usepackage[pdftex,%
	bookmarks=true,%
	bookmarksopen=true,%
	bookmarksopenlevel=5,%
	colorlinks=true,%
	pdftitle=Arxiv\ \todayiso:\ K-Independent\ Percolation\ On\ Trees,%
	pdfauthor=Temmel,%
]{hyperref}
\input{COMMON/latex/macro_general}
\newcommand{\Then}{\,\Rightarrow\,}

\newcommand{\Iff}{\,\Leftrightarrow\,}
\newcommand{\ForAll}{\forall\,}
\newcommand{\Exists}{\exists\,}

\newcommand{\FirstAlign}{\quad\,\,}

\newcommand{\NatNum}{\mathbb{N}}
\newcommand{\NatNumZero}{\NatNum_0} 
\newcommand{\IntNum}{\mathbb{Z}}

\newcommand{\RealNum}{\mathbb{R}}
\newcommand{\RealNumPlus}{\RealNum_{+}}

\newcommand{\Set}[1]{{\{#1\}}}

\newcommand{\IntegerSet}[1]{{[#1]}}

\newcommand{\Cardinality}[1]{{|#1|}}

\newcommand{\Modulus}[1]{|#1|}
\newcommand{\Indicator}[1]{\,\mathbb{I}_{#1}}

\newcommand{\BigCeiling}[1]{{\left\lceil#1\right\rceil}}
\newcommand{\FracCeiling}[2]{{\BigCeiling{\frac{#1}{#2}}}}

\newcommand{\Graph}[1]{G(#1)}

\newcommand{\Vertices}[1]{V(#1)}
\newcommand{\Neighbours}[1]{\mathcal{N}(#1)}

\newcommand{\GraphMetricSymbol}{d}
\newcommand{\GraphMetric}[2]{\GraphMetricSymbol(#1,#2)}

\newcommand{\ConnectedTo}{{\,\leftrightarrow\,}}

\newcommand{\CutsetsOf}[1]{{\Pi(#1)}}

\newcommand{\AlmostSureLy}{\text{a.s. \hspace{-0.4ex}}}
\newcommand{\Iid}{\text{i.i.d. \hspace{-0.4ex}}}

\newcommand{\Proba}{\mathbb{P}}
\newcommand{\Expect}{\mathbb{E}}

\newcommand{\MeasureSpaceSymbol}{\mathcal{M}}

\newcommand{\ProbabilityMeasureSpace}[1]{\MeasureSpaceSymbol_1(#1)}

\newcommand{\Algebra}{{\mathcal{A}}}

\newcommand{\Tree}{\mathbb{T}}
\newcommand{\NodeLevel}[1]{{l(#1)}}
\newcommand{\TreeLevel}[2]{{L(#1,#2)}}
\newcommand{\TreePath}[2]{{P(#1,#2)}}
\newcommand{\SubtreeRootedAt}[2]{{{#1}^{#2}}}

\newcommand{\Parent}[1]{{\mathfrak{p}(#1)}}

\newcommand{\TreeBoundary}[1]{{\partial #1}}
\newcommand{\Confluent}{{\curlywedge}}

\newcommand{\BranchingNumber}[1]{{br(#1)}}

\newcommand{\RenderName}[1]{#1} 
\newcommand{\MakeName}[2]{
	\newcommand{#1}{\RenderName{#2}}
}
\MakeName{\NameAaronson}{Aaronson}
\MakeName{\NameAaronsonEtAl}{\NameAaronson{} et al.}
\MakeName{\NameAaronsonGilatKeaneDeValk}{\NameAaronson{}, Gilat, Keane \& de Valk}
\MakeName{\NameAntalPisztora}{Antal \& Pisztora}
\MakeName{\NameBernoulli}{Bernoulli}
\MakeName{\NameBalister}{Balister}
\MakeName{\NameBollobas}{Bollob\'{a}s}
\MakeName{\NameBalisterBollobas}{\NameBalister{} \& \NameBollobas{}}
\MakeName{\NameBoltzmann}{Boltzmann}
\MakeName{\NameCayley}{Cayley}
\MakeName{\NameDijkstra}{Dijkstra}
\MakeName{\NameDobrushin}{Dobrushin}
\MakeName{\NameDobrushinLangfordRuelle}{\NameDobrushin{},Langford \& Ruelle}
\MakeName{\NameDurrett}{Durrett}
\MakeName{\NameErdos}{Erd\"{o}s}
\MakeName{\NameErdosLovasz}{\NameErdos{} \& \NameLovasz{}}
\MakeName{\NameFaris}{Faris}
\MakeName{\NameFernandez}{Fern\'{a}ndez}
\MakeName{\NameFernandezProcacci}{\NameFernandez{} \& \NameProcacci{}}
\MakeName{\NameFolner}{F\o{}lner}
\MakeName{\NameGaltonWatson}{Galton-Watson}
\MakeName{\NameGibbs}{Gibbs}
\MakeName{\NameGuttman}{Guttman}
\MakeName{\NameHausdorff}{Hausdorff}
\MakeName{\NameHessenberg}{Hessenberg}
\MakeName{\NameIverson}{Iverson}
\MakeName{\NameKolmogorov}{Kolmogorov}
\MakeName{\NameKotecky}{Koteck\'{y}}
\MakeName{\NameKoteckyPreiss}{\NameKotecky{} \& \NamePreiss{}}
\MakeName{\NameLebesgue}{Lebesgue}
\MakeName{\NameLiggett}{Liggett}
\MakeName{\NameLiggettEtAl}{\NameLiggett{} et al.}
\MakeName{\NameLiggettSchonmannStacey}{\NameLiggett{}, \NameSchonmann{} \& \NameStacey{}}
\MakeName{\NameLovaszLocalLemma}{\NameLovasz{} Local Lemma}
\MakeName{\NameLovasz}{Lov\'{a}sz}
\MakeName{\NameLyons}{Lyons}
\MakeName{\NameMathieu}{Mathieu}
\MakeName{\NameMayer}{Mayer}
\MakeName{\NameMiracleSole}{Miracle-Sol\'{e}}
\MakeName{\NamePenrose}{Penrose}
\MakeName{\NamePeres}{Peres}
\MakeName{\NamePreiss}{Preiss}
\MakeName{\NameProcacci}{Procacci}
\MakeName{\NameRamsey}{Ramsey}
\MakeName{\NameSchonmann}{Schonmann}
\MakeName{\NameScott}{Scott}
\MakeName{\NameScottSokal}{\NameScott{} \& \NameSokal{}}
\MakeName{\NameShearer}{Shearer}
\MakeName{\NameShearersMeasure}{\NameShearer{}'s measure}
\MakeName{\NameSokal}{Sokal}
\MakeName{\NameStacey}{Stacey}
\MakeName{\NameTarski}{Tarski}
\MakeName{\NameTemmel}{Temmel}
\MakeName{\NameTodo}{Todo}
\MakeName{\NameUrsell}{Ursell}
\MakeName{\NameVanHove}{van Hove}
\MakeName{\NameWoess}{Woess}
\MakeName{\NameYoung}{Young}
\MakeName{\NameZorn}{Zorn}

\usepackage{comment}
\specialcomment{NoteToSelf}
	{\begin{paragraph}{Note to self:}}
	{{\hfill$\Diamond$}\end{paragraph}\vspace{1em}}

\excludecomment{NoteToSelf}
\usepackage{amsthm}

\theoremstyle{plain}
\newtheorem{Thm}{Theorem}
\newtheorem{Lem}[Thm]{Lemma}
\newtheorem{Prop}[Thm]{Proposition}
\newtheorem{Cor}[Thm]{Corollary}

\theoremstyle{definition}

\newtheorem{Not}[Thm]{Notation}

\theoremstyle{definition}

\newtheorem{Mod}[Thm]{Model}

\theoremstyle{remark}

\newtheorem*{Rem}{Remark}

\theoremstyle{definition}

\theoremstyle{plain}

\theoremstyle{plain} 
\xyoption{arrow}
\xyoption{frame}
\xyoption{matrix}
\newcommand{\MxyZeroOneSwitch}[1]{
\xymatrix@C=#1 {
	Z & &{\dotso} &{\dotso} &{\dotso} &1 &1 &1 &1 &0 &1 &1 &1 &0 &{\dotso}\\
	X & &{\dotso} &0        &1        &0 &1 &1 &1 &0 &0 &1 &1 &0 &{\dotso}
	\ar "1,10"; "2,10"
	\ar "1,10"; "2,9"
	\ar "1,10"; "2,8"
	\ar@{-->} "2,8"; "1,8"
	\ar@{-->} "2,9"; "1,9"
	\ar@{-->} "2,10"; "1,11"
	\ar@{-->} "2,10"; "1,12"
	\save "1,8"."1,9" *\frm{--} \restore
	\save "1,11"."1,12" *\frm{--} \restore
	\save "2,8"."2,10" *\frm{-} \restore
	\save "1,10" *++\frm{o} \restore
}}
\newcommand{\MxyFuzzTwoIntNumShearer}[1]{
\xymatrix@C=#1 {
	Z & &{.} &{.} &{.} &0 &1 &1 &1 &1 &1 &0 &1 &1 &1 &0 &{.}\\
	X & &{.} &1   &1   &0 &1 &0 &1 &1 &1 &0 &0 &1 &1 &0 &{.}
	\ar "1,12"; "2,12"
	\ar "1,12"; "2,11"
	\ar "1,12"; "2,10"
	\ar@{-->} "2,10"; "1,10"
	\ar@{-->} "2,11"; "1,11"
	\ar@{-->} "2,12"; "1,13"
	\ar@{-->} "2,12"; "1,14"
	\save "1,10"."1,11" *\frm{--} \restore
	\save "1,13"."1,14" *\frm{--} \restore
	\save "2,10"."2,12" *\frm{-} \restore
	\save "1,12" *++\frm{o} \restore
	\ar "1,8"; "2,8"
	\ar "1,8"; "2,7"
	\ar "1,8"; "2,6"
	\save "2,6"."2,8" *\frm{-} \restore
	\save "1,8" *++\frm{o} \restore
}}

\xyoption{graph}
\newcommand{\MxyCanonicalModelSupportBB}{
\xygraph{
	&&{\circ}="n10"\\
	&{\Parent{v}}="p" &&{\circ}="n20"\\
	&{v}="v" &{\circ}="n30" &{\circ}="n31"\\
	{\circ}="n40" &{\circ}="n41" &{\circ}="n42" &{\circ}="n43"
	"n10" ( -@{-} "p" , -@{-} "n20")
	"p" ( -@{=} "v" ^{e} , -@{-} "n30")
	"n20" -@{-} "n31"
	"v" ( -@{-} "n40" , -@{-} "n41")
	"n30" -@{-} "n42"
	"n31" -@{-} "n43"
}}
\xyoption{graph}
\newcommand{\MxyPercolationKernelNormal}{
\xygraph{
	&o\\
	&*+{u=v\Confluent w} ="u"\\
	*+{v} ="v"&&t\\
	&&*+{w} ="w"
	"o" -@{-} "u"
	"u" -@{-} "v"
	"u" -@{-} "t" ^{\GraphMetric{v}{t}= (k\lor s)+1}
	"t" -@{-} "w" ^{\GraphMetric{t}{w}\ge 1}
}}
\newcommand{\MxyPercolationKernelDegenerated}{
\xygraph{
	&o\\
	&*+{u=v\Confluent w} ="u"\\
	*+{v} ="v" &&*+{t=w} ="t"
	"o" -@{-} "u"
	"u" -@{-} "v"
	"u" -@{-} "t" ^{\GraphMetric{v}{t}\le (k\lor s)+1}
}}

\newcommand{\BranchingNumberT}{{\BranchingNumber{\Tree}}}
\newcommand{\BranchingNumberTReciprocal}{{\frac{1}{\BranchingNumberT}}}

\newcommand{\Energy}[1]{{\mathcal{E}(#1)}}
\newcommand{\ExplicitHK}{{h_k}}
\newcommand{\ExplicitGK}{{g_k}}
\newcommand{\ExplicitFK}{{f_k}}

\newcommand{\FuzzKL}[1]{{[#1]_{(k)}}}
\newcommand{\FuzzKN}{\NatNum_{(k)}}
\newcommand{\FuzzKZ}{\IntNum_{(k)}}
\newcommand{\FuzzZeroZ}{\IntNum_{(0)}}

\newcommand{\PercolationClass}[3]{{\mathcal{C}_{#1}^{#2}(#3)}}
\newcommand{\PercolationClassRooted}[5]{{\mathcal{C}_{#1,#2}^{#3,#4}(#5)}}
\newcommand{\Percolation}{{\mathcal{P}}}
\newcommand{\PercolationCutNK}{{\Percolation^{cut(k,N)}}}
\newcommand{\PercolationMinimalK}[1]{{\Percolation_{#1}^{min(k)}}}
\newcommand{\PercolationCanonicalK}[1]{{\Percolation_{#1}^{can(k)}}}
\newcommand{\PercolationKernel}{{\kappa}}

\newcommand{\PowerFracDualK}{{\frac{k^k}{(k+1)^{(k+1)}}}}
\newcommand{\ShearerCriticalFunction}[1]{{\Xi_{#1}}}
\newcommand{\ShearerMeasure}[2]{{\mu_{#1,#2}}}

\newcommand{\TheXi}{\xi}

\newcommand{\PMax}[2]{{p_{max}^{#1}(#2)}}
\newcommand{\PMaxE}[1]{\PMax{#1}{E}}
\newcommand{\PMaxV}[1]{\PMax{#1}{V}}
\newcommand{\PMin}[2]{{p_{min}^{#1}(#2)}}
\newcommand{\PMinE}[1]{\PMin{#1}{E}}
\newcommand{\PMinV}[1]{\PMin{#1}{V}}

\newcommand{\PShearer}[1]{{p_{sh}^{#1}}}
\newcommand{\PShearerKL}[1]{{\PShearer{\FuzzKL{#1}}}}
\newcommand{\PShearerKN}{{\PShearer{\FuzzKN}}}
\newcommand{\PShearerKZ}{{\PShearer{\FuzzKZ}}}
\newcommand{\ResourceCriticalValuesImage}{COMMON/specific/treepercolation/criticalvalues}
\title{K-Independent Percolation on Trees}
\author{
	Pierre Mathieu
		\footnote{L.A.T.P., UMR-CNRS 6632, C.M.I., Université de Provence
39 rue Joliot Curie, 13453 Marseille cedex 13, France }
		\footnote{Email: Pierre.Mathieu@cmi.univ-mrs.fr}
	\hspace{1cm}
	Christoph Temmel
		\footnote{5030 Institut für Mathematische Strukturtheorie, Technische Universität Graz, Steyrergasse 30/III, 8010 Graz, Austria}
		\footnote{Email: temmel@math.tugraz.at}
}
\date{}
\begin{document}

\generalisodate{}
\maketitle{}

\begin{abstract}
Consider the class of $k$-independent bond or site percolations with parameter $p$ on a tree $\Tree$. We derive tight bounds on $p$ for both almost sure percolation and almost sure nonpercolation. The bounds are continuous functions of $k$ and the branching number of $\Tree$. This extends previous results by \NameLyons{} for the independent case ($k=0$) and by \NameBalisterBollobas{} for $1$-independent bond percolations. Central to our argumentation are moment method bounds à la \NameLyons{} supplemented by explicit percolation models à la \NameBalisterBollobas{}. An indispensable tool is the minimality and explicit construction of \NameShearersMeasure{} on the $k$-fuzz of $\IntNum$.
\end{abstract}





\newcommand{\Keywords}{
\begin{paragraph}{Keywords:}
  k-independent,
  k-dependent,
  tree percolation,
  critical value,
  percolation kernel,
  second moment method,
  \NameShearersMeasure{}.
\end{paragraph}
}

\newcommand{\Classification}{
\begin{paragraph}{MSC 2010:}
  82B43 (primary), 
  60K35, 
  82B20, 
  05C05. 
\end{paragraph}
}

\Classification{}

$\phantom{}$\vspace{1em}\hrule\vspace{1em}
This is an extended version of \cite{MathieuTemmel_kindependent} by the second author.
\vspace{1em}\hrule\vspace{1em}

\tableofcontents{}
\listoffigures{}

\section{Introduction}\label{sec:introduction}
If we regard percolation on a tree $\Tree$, then a natural question is which properties of the percolation and the tree determine the percolation behaviour. One is especially interested in bounds which are not particular to a specific model, but are valid for whole classes of models. The class of models we investigate are $k$-independent (also called $k$-dependent in the literature) site (bond) percolations with parameter $p$, i.e. the probability that a single vertex (edge) is open is $p$ and subsets of vertices (edges) are independent if their distance is greater than $k$. We look for bounds on the parameter $p$ which guarantee either \AlmostSureLy{} percolation or \AlmostSureLy{} nonpercolation.\\

\NameLyons{} \cite{Lyons_rw_percolation_tree} first treated this question in the case of independent percolation. He defined the branching number $\BranchingNumberT$ as a measure of the size of $\Tree$. Then he showed that it is the characteristic determining the critical probability for independent percolation (see theorem \ref{thm:valueOfCriticalValuesZero}), that is the parameter threshold at which nonpercolation switches to percolation.\\

A recent work by \NameBalisterBollobas{} \cite{BalisterBollobas_onepercolation} deals with the class of $1$-independent bond percolations (see theorem \ref{thm:valueOfCriticalValuesOneBond}). There are two continuous functions of the branching number which give tight bounds for \AlmostSureLy{} percolation and \AlmostSureLy{} nonpercolation of each model in this class.\\

In section \ref{sec:mainResults} we present our results: tight bounds for \AlmostSureLy{} percolation and \AlmostSureLy{} nonpercolation for every $k$. The bounds are again continuous functions of $\BranchingNumberT$, parametrized by $k$. They are the same for bond and site percolation. A core ingredient is a probability measure introduced by \NameShearer{} \cite{Shearer_problem}, which has certain nice minimizing properties, reviewed in section \ref{sec:shearerDefinitionAndGeneralProperties}. We construct it explicitly on the $k$-fuzz of $\IntNum$ in section \ref{sec:shearerKFuzzZ} and show that it is a $(k+1)$-factor. \NameShearersMeasure{} minorizes the probability of having an open path of $k$-independent \NameBernoulli{} rvs. This property is already exploited implicitly in the work of \NameBalisterBollobas{}. We make this argument explicit by using moment method and capacity arguments motivated by \NameLyons{}' proof \cite{Lyons_rw_percolation_tree,Lyons_rw_capacity}, supplemented with explicit percolation models inspired by \NameBalisterBollobas{}' work \cite{BalisterBollobas_onepercolation}.
\section{Setup and previous results}
\label{sec:setupAndPreviousResults}
Let $G:=(V,E)$ be a graph. For every subset $H$ of vertices and/or edges of $G$ denote by $\Vertices{H}$ the \emph{vertices induced by} $H$ and by $\Graph{H}$ the \emph{subgraph of $G$ induced by} $H$. We have the \emph{geodesic graph distance} $\GraphMetricSymbol$ on both vertices and edges, extended naturally to sets of them. Define the equivalence relation $v\ConnectedTo w$ describing \emph{connectedness} on $G$. We denote by $\Neighbours{v}$ the \emph{neighbours} of a vertex $v$. The \emph{$k$-fuzz} (or \emph{$k^{th}$ power}) of $G$ is the graph $(V,E')$, where $E'$ consists of all distinct pairs of vertices with distance less than or equal to $k$ in $G$.\\

We primarily work on a \emph{locally finite tree} $\Tree:=(V,E)$. We consider it to be \emph{infinite}, unless explicitly stated otherwise. Between two nodes $v$ and $w$ we have the unique \emph{geodesic path} $\TreePath{v}{w}$. For the following definitions root $\Tree$ at the \emph{root} $o$ and visualize the tree spreading out downwards from the root. Define the \emph{level} $\NodeLevel{v}:=\GraphMetric{o}{v}$ of a node $v$ and let $\TreeLevel{\Tree}{n}:=\Set{v:\NodeLevel{v}=n}$ be the $n^{th}$ level of $\Tree$. \emph{Downpaths} and \emph{-rays} are finite and infinite geodesics, which start at some vertex $v$ and go downwards, thereby avoiding all ancestors of $v$, respectively. Denote the \emph{boundary} of $\Tree$ by $\TreeBoundary{\Tree}$, which is the set of all ends of $\Tree$, identified with the set of all downrays of $\Tree$ starting at $o$. For all nodes $v\in V\setminus\Set{o}$ there is a unique \emph{parent} denoted by $\Parent{v}$. The \emph{confluent} $v\Confluent w$ is the last common ancestor of two distinct nodes $v$ and $w$. Define a \emph{minimal vertex cutset} $\Pi$ to be a finite set of vertices containing no ancestors of itself and delineating a connected component containing $o$. Denote by $\CutsetsOf{o}$ the \emph{set of all minimal vertex cutsets} of $o$. Finally let $\SubtreeRootedAt{\Tree}{v}$ be the \emph{induced subtree of $\Tree$ rooted at $v$}.\\

Furthermore we abbreviate $\Set{1,\dotsc,n}$ by $\IntegerSet{n}$. As a convention we interpret $\IntegerSet{0}:=\emptyset$, $0^0:=1$, empty products as $1$ and empty sums as $0$.\\

\begin{NoteToSelf}
Later definitions:
\begin{itemize}
  \item diameter before theorem \ref{thm:diameterBoundedThreshold}
  \item independent set in \ref{sec:shearerDefinitionAndGeneralProperties}
  \item $k$-fuzz in \ref{sec:shearerKFuzzZ}
  \item factor and stationary in section \ref{sec:shearerKFuzzZ}
\end{itemize}
Recalls:
\begin{itemize}
  \item confluent in \ref{sec:percolationKernelEstimates}
\end{itemize}
\end{NoteToSelf}

Recall that a \emph{bond} and \emph{site} \emph{percolation} on a graph $G:=(V,E)$ is an rv taking values in $\Set{0,1}^E$ and $\Set{0,1}^V$ respectively. A percolation \emph{percolates} iff it induces an infinite percolation cluster (connected component) in $G$ with nonzero probability.\\

We investigate percolations on a tree $\Tree:=(V,E)$ and look for properties of the percolation and the tree influencing the percolation behaviour. For $k\in\NatNumZero$ we consider the class of \emph{$k$-independent site percolations with parameter $p$} on $\Tree$, denoted by $\PercolationClass{p}{k}{V}$. A site percolation $Z:=\Set{Z_v}_{v\in V}$ has parameter $p$ iff the probability that a single site is open equals $p$. For $W\subseteq V$ let $Z_W:=\Set{Z_v}_{v\in W}$. The site percolation $Z$ is \emph{$k$-independent} iff
\begin{equation}\label{eq:kIndependence}
	\ForAll U, W\subset V:\quad
	\GraphMetric{U}{W}>k \Then Z_U\text{ is independent of }Z_W\,,
\end{equation}
that is events on subsets at distance greater than $k$ are independent. Independence is synonymous with $0$-independence. The present paper investigates bounds on the parameter $p$ guaranteeing either \AlmostSureLy{} percolation or \AlmostSureLy{} nonpercolation for the whole class. We define the \emph{critical values}
\begin{subequations}\label{eq:criticalValues}
\begin{align}
	\label{eq:pMaxDef}
	\PMaxV{k} &:= \inf\Set{p\in [0,1]:
		\ForAll\Percolation\in\PercolationClass{p}{k}{V}:
		\Percolation\text{ percolates}
		}\\
	\label{eq:pMinDef}
	\PMinV{k} &:= \inf\Set{p\in [0,1]:
		\Exists\Percolation\in\PercolationClass{p}{k}{V}:
		\Percolation\text{ percolates}
		}\,.
\end{align}
\end{subequations}

Analogously we define the class $\PercolationClass{p}{k}{E}$ of $k$-independent bond percolations with parameter $p$ on $\Tree$ and critical values $\PMaxE{k}$ and $\PMinE{k}$.\\

A \emph{$\lambda$-flow} on $\Tree$ is a function $f: V\mapsto\RealNumPlus$ such that
\begin{equation}\label{eq:lambdaFlow}
	\ForAll v\in V:\quad
	0\le f(v)=\sum_{w:\,\Parent{w}=v} f(w)\le \lambda^{-\NodeLevel{v}}\,.
\end{equation}
\NameLyons{} \cite{Lyons_rw_percolation_tree} introduced the \emph{branching number} $\BranchingNumberT$ as a measure of the size of a tree $\Tree$:
\begin{equation}\label{eq:branchingNumber}
	\begin{aligned}
		\BranchingNumberT
			:=& \sup\Set{\lambda \ge 1: \Exists \text{ nonzero } \lambda\text{-flow on }\Tree}\\
			=& \sup\Set{\lambda \ge 1: \inf_{\Pi\in\CutsetsOf{o}}
				\sum_{v\in\Pi} \lambda^{-\NodeLevel{v}}>0}\,.
	\end{aligned}
\end{equation}
The duality in \eqref{eq:branchingNumber} is due to the max-flow min-cut theorem on infinite graphs \cite{FordFulkerson_flows}. The branching number $\BranchingNumberT$ is independent of the choice of $o$ and equals the \emph{exponential of the Hausdorff dimension of the boundary $\TreeBoundary{\Tree}$} of $\Tree$ \cite[section 1.8]{LyonsPeres_prob_tree_network}. Throughout this paper we assume $\BranchingNumberT$ to be finite.\\

\begin{NoteToSelf}
In case PM ever asks again: the outer $\sup$s in \eqref{eq:branchingNumber} are the same, just the expression inside uses the max-flow, min-cut duality.
\end{NoteToSelf}

The first known result is due to \NameLyons{} \cite{Lyons_rw_percolation_tree}, where he characterized the \emph{critical value of independent percolation} ($k=0$) in terms of $\BranchingNumberT$:

\begin{Thm}[{\cite[theorem 6.2]{Lyons_rw_percolation_tree}}]\label{thm:valueOfCriticalValuesZero}
\begin{equation}\label{eq:valueOfCriticalValuesZero}
	\PMinV{0}=\PMinE{0}=\PMaxV{0}=\PMaxE{0}=\BranchingNumberTReciprocal\,.
\end{equation}
\end{Thm}

In the independent case the critical values coincide, since for fixed $p$ there is only one percolation. \NameLyons{}' proof is based on moment methods and capacity estimates of percolation kernels. In general, a percolation is \emph{quasi-independent} \cite[section 2.4]{Lyons_rw_capacity} iff, using the notation from figure \ref{fig:percolationKernel} on page \pageref{fig:percolationKernel} with $u:=v\Confluent w$ the confluent of $v$ and $w$, we have an $M>0$ such that for all $v$ and $w$
\begin{equation}\label{eq:quasiIndependent}
	\Proba(o\ConnectedTo v,o\ConnectedTo w|o\ConnectedTo u)
	\le M \Proba(o\ConnectedTo v|o\ConnectedTo u)
	\Proba(o\ConnectedTo w|o\ConnectedTo u)\,.
\end{equation}
Equivalently this majorizes the \emph{percolation kernel} \eqref{eq:percolationKernel} by
\begin{equation}\label{eq:percolationKernelQuasiIndependentMajoration}
	\PercolationKernel(v,w):=\frac%
		{\Proba(o\ConnectedTo v,o\ConnectedTo w)}
		{\Proba(o\ConnectedTo v)\Proba(o\ConnectedTo s)}
	\le \frac{M}{\Proba(o\ConnectedTo u)}\,.
\end{equation}
This way \NameLyons{} \cite[section 2.4]{Lyons_rw_capacity} used the weighted second moment method to get bounds for the probability of the percolation reaching subsets of $\TreeBoundary{\Tree}$ in terms of their \emph{capacity}, extending the independent case in \cite{Lyons_rw_percolation_tree}.\\

In a recent work, \NameBalisterBollobas{} \cite{BalisterBollobas_onepercolation} deal with the class of $1$-independent bond percolations:

\begin{Thm}[\cite{BalisterBollobas_onepercolation}]\label{thm:valueOfCriticalValuesOneBond}
\begin{subequations}\label{eq:valueOfCriticalValuesOneBond}
\begin{align}
	\label{eq:valueOfPMinOneBond}
	\PMinE{1}&=\frac{1}{\BranchingNumberT^2}\\
	\label{eq:valueOfPMaxOneBond}
	\PMaxE{1}&=\begin{cases}
			1-\frac{\BranchingNumberT-1}{\BranchingNumberT^2}
				&\text{if }\BranchingNumberT\le 2\\
			\frac{3}{4}
				&\text{if }\BranchingNumberT\ge 2.
	    \end{cases}
\end{align}
\end{subequations}
\end{Thm}

Their proof strategy for $\PMinE{1}$ is based on the first moment method and a simple explicit model. We generalize it rather straightforwardly to higher $k$ in section \ref{sec:determiningPMin}. Their proof for $\PMaxE{1}$ on the other hand combines a so-called canonical model (discussed in section \ref{sec:canonical}) with several short and elementary inductive proofs (see \cite{Temmel_kindependent}). For every $p\ge\frac{3}{4}$ this canonical model minimizes the probability to percolate. They implicitly retrace the weighted second moment method, percolation kernel capacity estimates based on $\lambda$-flows and the minimizing property \eqref{eq:shearerMinimalityOVOEP}, \eqref{eq:shearerKFuzzZEscapingOVOEPMinoration} of \NameShearersMeasure{} \cite{Shearer_problem} on $\IntNum$. Alas this inductive approach exploits a few particularities of the case $k=1$, which we have not been able to abstract from.

\section{Main results}\label{sec:mainResults}
Our principal result in the setting of section \ref{sec:setupAndPreviousResults} (see also figure \ref{fig:principalResults} on page \pageref{fig:principalResults}) is:

\begin{Thm}\label{thm:valueOfCriticalValuesK} $\ForAll k\in\NatNumZero$:
\begin{subequations}\label{eq:valueOfCriticalValuesK}
\begin{align}
	\label{eq:valueOfPMinK}
	\PMinV{k}=\PMinE{k}&=\frac{1}{\BranchingNumberT^{k+1}}\\
	\label{eq:valueOfPMaxK}
	\PMaxV{k}=\PMaxE{k}&=\begin{cases}
			1-\frac{\BranchingNumberT-1}{\BranchingNumberT^{k+1}}
				&\text{if }\BranchingNumberT\le\frac{k+1}{k}\\
			1-\PowerFracDualK
				&\text{if }\BranchingNumberT\ge\frac{k+1}{k}\,,
	    \end{cases}
\end{align}
where we interpret $\frac{1}{0}:=\infty$ in the case $k=0$.
\end{subequations}
\end{Thm}

This theorem is a corollary of the more general theorem \ref{thm:valueOfRootedCriticalValues} upon setting $s=k$ and verifying that (un)rooting a percolation does not change its percolation behaviour (see section \ref{sec:generalTreePercolationTools}).\\

First we narrow down the definition of the percolation classes we work on. Let $\PercolationClassRooted{p}{o}{k}{s}{V}$ be the class of \emph{rooted site percolations with parameter $p$} on $\Tree$ which are \emph{$k$-independent along downrays from $o$} and \emph{$s$-independent elsewhere}, that is among vertices not on the same downray. We define the \emph{rooted critical values} as
\begin{subequations}\label{eq:rootedCriticalValues}
\begin{align}
	\label{eq:rootedpMaxDef}
	\PMaxV{k,s} &:= \inf\Set{p\in [0,1]:
		\ForAll o\in V:
		\ForAll\Percolation\in\PercolationClassRooted{p}{o}{k}{s}{V}:
		\Percolation\text{ percolates}
		}\\
	\label{eq:rootedpMinDef}
	\PMinV{k,s} &:= \inf\Set{p\in [0,1]:
		\Exists o\in V:
		\Exists\Percolation\in\PercolationClassRooted{p}{o}{k}{s}{V}:
		\Percolation\text{ percolates}
		}\,.
\end{align}
\end{subequations}

Analogously we define the class $\PercolationClassRooted{p}{o}{k}{s}{E}$ of $k,s$-independent, rooted bond percolations with parameter $p$ on $\Tree$ and the critical values $\PMinE{k,s}$ and $\PMaxE{k,s}$. Define the function
\begin{equation}\label{eq:explicitGK}
	\ExplicitGK:\quad
		[1,\infty]\to]0,1[\qquad
		y \to 1-\frac{y-1}{y^{k+1}}
\end{equation}
and the value
\begin{equation}\label{eq:pShearerKFuzzZDefinition}
	\PShearerKZ:=1-\PowerFracDualK\,.
\end{equation}
We reveal their motivation in proposition \ref{prop:pShearerKFuzzZ} and \ref{prop:secondMomentBound} respectively. Our main result determines the critical values \eqref{eq:rootedCriticalValues}:

\begin{Thm}\label{thm:valueOfRootedCriticalValues}
$\ForAll k, s\in\NatNumZero$:
\begin{subequations}\label{eq:valueOfRootedCriticalValues}
\begin{align}
	\label{eq:valueOfRootedPMinT}
	\PMinV{k,s}=\PMinE{k,s}&=\frac{1}{\BranchingNumberT^{k+1}}\\
	\label{eq:valueOfRootedPMaxT}
	\PMaxV{k,s}=\PMaxE{k,s}&=\begin{cases}
		1-\frac{\BranchingNumberT-1}{\BranchingNumberT^{k+1}}
		=\ExplicitGK(\BranchingNumberT)
			&\text{if }\BranchingNumberT\le\frac{k+1}{k}\\
		1-\PowerFracDualK=\PShearerKZ
			&\text{if }\BranchingNumberT\ge\frac{k+1}{k}\,,\\
	\end{cases}
\end{align}
where we interpret $\frac{1}{0}:=\infty$ in the case $k=0$.
\end{subequations}
\end{Thm}

We give the proof in section \ref{sec:proofs} and a plot of the results \eqref{eq:valueOfRootedCriticalValues} in figure \ref{fig:principalResults}.\\

\begin{figure}[!htbp]
\begin{center}
\includegraphics[width=12cm]{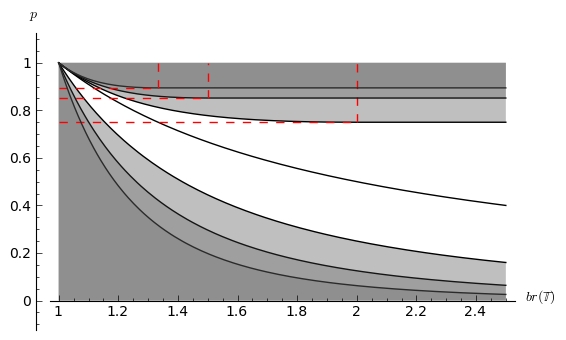}
\end{center}
\caption[Plot of principal results]{(Colour online) The curves of $\PMaxV{k,s}$ and $\PMinV{k,s}$ for $k\in\Set{0,1,2,3}$ and branching numbers in $[1,2.5]$ delimit increasingly shaded regions. The dashed red lines mark the points $\left(\frac{k+1}{k},\PShearerKZ\right)$ for $k\ge 1$, where the behaviour of $\PMaxV{k,s}$ changes.}
\label{fig:principalResults}
\end{figure}

The critical values \eqref{eq:valueOfRootedCriticalValues} are \emph{independent of} the root $o$, the elsewhere-dependence range $s$ and whether we regard bond or site percolation. A change of the root from $o$ to $o'$ turns a $k,s$-independent percolation at worst into a $(k\lor s),(k\lor s)$-independent percolation. Upon closer inspection one sees that this concerns only elements contained in the ball of radius $\GraphMetric{o}{o'}+(k\lor s)$ around $o'$. They are finitely many and one can ignore them as percolation is a tail-event (see the adaption of \NameKolmogorov{}'s zero-one law in lemma \ref{lem:kolmogorovZeroOneLocallyDependent}), hence the percolation essentially remains $k,s$-independent. The independence of the parameter $s$ is a consequence of the use of the moment methods, which only take into account the structure of a rooted percolation along downrays. There is a \emph{bijection} from $E$ to $V\setminus\Set{o}$ mapping an edge to its endpoint further away from the root $o$. This implies that $\PercolationClassRooted{p}{o}{k}{s}{E}\subseteq\PercolationClassRooted{p}{o}{k}{s+1}{V}$. Furthermore we have $\PercolationClassRooted{p}{o}{k}{0}{V}=\PercolationClassRooted{p}{o}{k}{1}{V}=\PercolationClassRooted{p}{o}{k}{0}{E}$ for $k\ge 1$ and $\PercolationClassRooted{p}{o}{0}{0}{V}=\PercolationClassRooted{p}{o}{0}{0}{E}$ as $\PercolationClassRooted{p}{o}{0}{1}{V}=\emptyset$. This allows the interpretation of our explicit site percolation models (models \ref{mod:canonical}, \ref{mod:cutup} and \ref{mod:minimal}) as $k,0$-independent bond percolation models. Hence we \emph{focus exclusively on site percolations} for the remainder of this paper.\\

We can generalize the single parameter $s$ to a family of finite and unbounded dependency parameters $\vec{s}:=\Set{s_v}_{v\in V}$. Then the upper bound on $\PMaxV{k,\vec{s}}$ in proposition \ref{prop:secondMomentBound} does not hold anymore. See the counterexample in model \ref{mod:multiplex} and proposition \ref{prop:multiplexFacts}. The lower bound on $\PMinV{k,\vec{s}}$ in propositions \ref{prop:firstMomentBound} and hence the value of $\PMinV{k,\vec{s}}$ stay valid under these less restrictive conditions and even for $s=\infty$, though.\\

We determine the critical values by a two-pronged approach. General bounds follow from a direct application of moment arguments \cite[sections 5.2/5.3]{LyonsPeres_prob_tree_network} and capacity estimates of percolation kernels \cite[section 1.9]{Lyons_rw_capacity}. In section \ref{sec:connectionWithQuasiIndependence} we show that in every instance where we apply the second moment method our $k,s$-independent percolations are quasi-independent \eqref{eq:quasiIndependent}. Analysis of a number of explicit percolation models (models \ref{mod:canonical}, \ref{mod:cutup} and \ref{mod:minimal}) renders the bounds tight. All explicit models are in the class $\PercolationClassRooted{p}{o}{k}{0}{V}$ and invariant under automorphisms of the rooted tree.\\

\NameShearersMeasure{} \cite{Shearer_problem} on the $k$-fuzz of $\IntNum$ (section \ref{sec:shearerKFuzzZ}) minimizes the conditional probability of the event ``open for $m$ more steps $|$ open for $n$ steps'' along a path of $k$-independent \NameBernoulli{} random variables (see \eqref{eq:shearerMinimalityOVOEP}). Our novel contribution is an explicit construction of \NameShearersMeasure{} on the $k$-fuzz of $\IntNum$ (model \ref{mod:shearerKFuzzZ}) as a $(k+1)$-factor for $p\ge\PShearerKZ$ via a zero-one switch (\eqref{eq:pShearerKFuzzZ} and figure \ref{fig:shearer2FuzzZ}), by reinterpreting calculations from \NameLiggettEtAl{} \cite[corollary 2.2]{LiggettSchonmannStacey_domination}. From the detailed knowledge about \NameShearersMeasure{} on the $k$-fuzz of $\IntNum$ we derive uniform bounds on the percolation kernel over the whole class $\PercolationClassRooted{p}{o}{k}{s}{V}$, leading to $\PMaxV{k,s}$.\\

A \emph{back-of-the-envelope derivation of the critical values} \eqref{eq:rootedCriticalValues} goes as follows: The simplest infinite rooted tree is a single ray isomorph to $\NatNum$. Let $Z:=(Z_n)_{n\in\NatNum}$ be a collection $k$-independent Bernoulli($p$) rvs on $\NatNum$. We have
\begin{equation}\label{eq:intuitiveComparisonRay}
	\TheXi^n\le\Proba(Z_{\IntegerSet{n}}=\vec{1})\le \eta^n\,,
\end{equation}
where the left inequality holds for $p\ge\PShearerKZ$ with the relation $p=1-\TheXi^k(1-\TheXi)$, thanks to \NameShearer{} (see section \ref{sec:shearerMeasure}), and the right one always with the relation $\eta^{k+1}=p$, thanks to $k$-independence. Root $\Tree$ and suppose that \eqref{eq:intuitiveComparisonRay} carries over to $k,s$-independent percolation with parameter $p$. Hence we have a comparison with two independent models with parameters $\TheXi$ and $\eta$, that is
\begin{equation}\label{eq:intuitiveComparisonPercolation}
	\Proba_\TheXi(\text{percolates})\le\Proba_p(\text{percolates})
	\quad\text{and}\quad
	\Proba_p(\text{percolates})\le\Proba_\eta(\text{percolates})\,,
\end{equation}
where the left inequality holds for $p\ge\PShearerKZ$. Plugging in $\BranchingNumberTReciprocal$, the critical value for independent percolation \eqref{eq:valueOfCriticalValuesZero}, for $\TheXi$ and $\eta$ we get the $\ExplicitGK$ part of $\PMaxV{k,s}$ for $\BranchingNumberT\le\frac{k+1}{k}$ and $\PMinV{k,s}$ for all $\BranchingNumberT$.\\

This comparison with two independent models in the last paragraph is solely in terms of the probability to percolate. We have no direct relation between the clusters (like a coupling between the percolations) and in particular no stochastic domination (see section \ref{sec:commentOnStochasticDomination}).\\

Already in the independent case \cite[section 5]{LyonsPeres_prob_tree_network} the percolation behaviour at $p=\BranchingNumberTReciprocal$ depends on additional properties of the tree. This stays the same for $\PMinV{k,s}$ and the $\ExplicitGK$ part of $\PMaxV{k,s}$. It is not so for $p=\PShearerKZ$ and $\BranchingNumberT>\frac{k+1}{k}$: here proposition \ref{prop:secondMomentBound} asserts that all $\Percolation\in\PercolationClassRooted{\PShearerKZ}{o}{k}{s}{V}$ percolate.\\

Recall that the \emph{diameter} of a percolation cluster is the length of the longest geodesic path contained in it. We call a percolation \emph{diameter bounded} if its percolation cluster diameters are \AlmostSureLy{} bounded, i.e.
\begin{equation}\label{eq:diameterBounded}
  \Exists D\in\NatNum:\quad
  \Proba(\sup\Set{\operatorname{diam}(C):C\text{ open cluster in }\Percolation}
  \le D)=1\,.
\end{equation}
The $\PShearerKZ$-line admits another interpretation in terms of cluster diameters:
\begin{Thm}\label{thm:diameterBoundedThreshold}
For each $\varepsilon>0$ there exist $p\in]\PShearerKZ-\varepsilon,\PShearerKZ[$ and $\Percolation\in\PercolationClassRooted{p}{o}{k}{0}{V}$ such that $\Percolation$ is diameter bounded. If $p\ge\PShearerKZ$, then all percolations in $\PercolationClassRooted{p}{o}{k}{0}{V}$ are not diameter bounded.
\end{Thm}
\section{\NameShearersMeasure{}}
\label{sec:shearerMeasure}
Throughout this section we assume $q:=1-p$.
\subsection{Definition and general properties}
\label{sec:shearerDefinitionAndGeneralProperties}
The graph $G:=(V,E)$ is a \emph{dependency graph} of a random field $Z:=\Set{Z_v}_{v\in V}$ iff
\begin{equation}\label{eq:strongDependencyGraph}
	\ForAll A,B\subset V:\quad
	\GraphMetric{A}{B}>1\Then Z_A\text{ is independent of }Z_B\,,
\end{equation}
that is non-adjacent subsets index independent subfields. The random field $Z$ may have several different dependency graphs \cite[section 4.1]{ScottSokal_repulsive}, in particular one can always add edges. A question which arose naturally in the context of the probabilistic method \cite{ErdosLovasz_hypergraphs} is: If we take $Z$ to be a \NameBernoulli{} random field with parameter $p$ and dependency graph $G$, what are the parameters $p$ for which we can guarantee that $\Proba(Z_V=\vec{1})>0$?\\

\NameShearer{} \cite{Shearer_problem} answered this question for finite $G$. He defined \emph{\NameShearer{}'s (signed) measure} $\ShearerMeasure{G}{p}$ on set $\Set{0,1}^V$ by setting the marginals \eqref{eq:shearerZeroMarginalDefinition} and constructing the other events by the \emph{inclusion-exclusion principle} \eqref{eq:shearerInclusionExclusion}:
\begin{subequations}\label{eq.shearersMeasure}
\begin{align}
	\label{eq:shearerZeroMarginalDefinition}
	\ForAll B\subseteq V:\quad&
	\ShearerMeasure{G}{p}(Y_B = \vec{0}):=
	\begin{cases}
		q^{\Cardinality{B}} &B\text{ independent}\\
		0 &B\text{ not independent}
	\end{cases}\\
	\label{eq:shearerInclusionExclusion}
	\ForAll B\subseteq V:\quad&
	\ShearerMeasure{G}{p}(Y_B = \vec{0},Y_{V\setminus B}=\vec{1})
	:=\sum_{\substack{B\subseteq T\subseteq V\\T\text{ independent}}}
	(-1)^{\Cardinality{T}-\Cardinality{B}} q^{\Cardinality{T}}\,.
\end{align}
\end{subequations}
Recall that an \emph{independent set of vertices} (in the graph-theoretic sense) contains no adjacent vertices. It is the second part of \eqref{eq:shearerZeroMarginalDefinition}, assigning zero probability to every realization with adjacent $0$s, that renders \NameShearersMeasure{} special among all measures with parameter $p$ and dependency graph $G$. Define the \emph{critical function}
\begin{equation}\label{eq:criticalShearerFunction}
	\ShearerCriticalFunction{G}(p)
	:= \ShearerMeasure{G}{p}(Y_V=\vec{1})
	= \sum_{\substack{T\subseteq V\\T\text{ independent}}}
		(-q)^{\Cardinality{T}}\,.
\end{equation}
It satisfies a \emph{fundamental identity}: $\ForAll v\in V, v\not\in W\subsetneq V, p\in[0,1]:$
\begin{equation}\label{eq:criticalShearerFunctionFundamentalIdentity}
	\quad
		\ShearerCriticalFunction{\Graph{W\uplus\Set{v}}}(p)
		=\ShearerCriticalFunction{\Graph{W}}(p)
		-q\,\ShearerCriticalFunction{\Graph{W\setminus\Neighbours{v}}}(p)\,,
\end{equation}
derived by discriminating between independent sets containing $v$ and those not. \NameShearersMeasure{} is a priori signed and only becomes a probability measure starting at a \emph{critical value} \cite[theorem 4.1 and proposition 2.18]{ScottSokal_repulsive}
\begin{equation}\label{eq:pShearerFinite}
	\PShearer{G}
	:= \max\Set{p: \ShearerCriticalFunction{G}(p)\le 0}
	= \min\Set{p: \ShearerMeasure{G}{p}\text{ is a probability measure}}\,.
\end{equation}
We emphasize that $\ShearerCriticalFunction{G}(\PShearer{G})=0$. For $p\ge\PShearer{G}$ the critical function $\ShearerCriticalFunction{G}(p)$ is the strictly monotone increasing probability that our realization \emph{contains only $1$s} \cite[proposition 2.18]{ScottSokal_repulsive}. The key property of \NameShearer{}'s probability measure is:
\begin{Lem}[\cite{Shearer_problem}]
\label{lem:shearerMinimality}
Let $Z$ be a random \NameBernoulli{} field with parameter $p\ge\PShearer{G}$ and dependency graph $G$. Let $Y$ be $\ShearerMeasure{G}{p}$-distributed. Then $\ForAll W\subseteq V$:
\begin{subequations}\label{eq:shearerMinimality}
\begin{equation}\label{eq:shearerMinimalityCriticalFunction}
	\Proba(Z_W=\vec{1})
	\ge\ShearerMeasure{G}{p}(Y_W=\vec{1})
	=\ShearerCriticalFunction{\Graph{W}}(p)
	\ge 0
\end{equation}
and $\ForAll W\subseteq \widetilde{W}\subseteq V$: if $\ShearerCriticalFunction{\Graph{W}}(p)>0$, then
\begin{equation}\label{eq:shearerMinimalityOVOEP}
	\Proba(Z_{\widetilde{W}}=\vec{1}|Z_{W}=\vec{1})
	\ge \ShearerMeasure{G}{p}(Y_{\widetilde{W}}=\vec{1}|Y_{W}=\vec{1})
	= \frac{\ShearerCriticalFunction{\Graph{\widetilde{W}}}(p)}%
	       {\ShearerCriticalFunction{\Graph{W}}(p)}
	\ge 0\,.
\end{equation}
\end{subequations}
\end{Lem}
\begin{proof}
It suffices to prove \eqref{eq:shearerMinimalityOVOEP} inductively for one-vertex extensions with $\widetilde{W}=W\uplus\Set{v}$. We prove \eqref{eq:shearerMinimality} jointly by induction over the cardinality of $W$. The induction base for $W=\Set{w}$ is:
\begin{equation*}
   \Proba(Z_w=1)
 = p
 = \ShearerMeasure{G}{p}(Y_w=1)
 =\ShearerCriticalFunction{(\Set{w},\emptyset)}(p)\,.
\end{equation*}
Induction step $W\to \widetilde{W}$: suppose that $\ShearerMeasure{G}{p}(Y_W=\vec{1})=0$. Hence also $\ShearerMeasure{G}{p}(Y_{\widetilde{W}}=\vec{1})=0$ and \eqref{eq:shearerMinimalityCriticalFunction} holds trivially. If $\ShearerMeasure{G}{p}(Y_{W}=\vec{1})>0$, then $\Proba(Z_{W}=\vec{1})>0$ by the induction hypothesis. Let $W\cap\Neighbours{v}=:\Set{w_1,\dotsc,w_m}$ and $W_i:=W\setminus\Set{w_i,\dotsc,w_m}$. If $m=0$, then we revert to the equality in the induction base. If $m\ge 1$ then
\begin{align*}
	&\FirstAlign\Proba(Z_v=1|Z_W=\vec{1})\\
	&=\frac{\Proba(Z_v=1,Z_W=\vec{1})}{\Proba(Z_W=\vec{1})}\\
	&\ge \frac%
		{\Proba(Z_W=\vec{1})-q\,\Proba(Z_{W\setminus\Neighbours{v}}=\vec{1})}
		{\Proba(Z_W=\vec{1})}
		&\text{as $Z$ has dependency graph $G$ \eqref{eq:strongDependencyGraph}}\\
	&= 1 - \frac{q}{\prod_{i=1}^m \Proba(Z_{w_i}=\vec{1}|Z_{W_i}=\vec{1})}\\
	&\ge 1 - \frac{q}{\prod_{i=1}^m
		\ShearerMeasure{G}{p}(Y_{w_i}=\vec{1}|Y_{W_i}=\vec{1})}
		&\text{induction hypothesis as $\Cardinality{W_i}<\Cardinality{W}$ }\\
	&= \frac%
		{\ShearerMeasure{G}{p}(Y_W=\vec{1})-
		 q\,\ShearerMeasure{G}{p}(Y_{W\setminus\Neighbours{v}}=\vec{1})}
		{\ShearerMeasure{G}{p}(Y_W=\vec{1})}\\
	&=\frac{\ShearerMeasure{G}{p}(Y_v=1,Y_W=\vec{1})}{\ShearerMeasure{G}{p}(Y_W=\vec{1})}
		&\text{using the fundamental identity \eqref{eq:criticalShearerFunctionFundamentalIdentity}}\\
	&=\ShearerMeasure{G}{p}(Y_v=1|Y_W=\vec{1})\,.
\end{align*}
This proves \eqref{eq:shearerMinimalityOVOEP}. To obtain \eqref{eq:shearerMinimalityCriticalFunction} it suffices to see that
\begin{multline*}
  \Proba(Z_{\widetilde{W}}=\vec{1})
  =\Proba(Z_v=1|Z_W=\vec{1})\Proba(Z_W=\vec{1})\\
  \ge\ShearerMeasure{G}{p}(Y_v=1|Y_W=\vec{1})\ShearerMeasure{G}{p}(Y_W=\vec{1})
  =\ShearerMeasure{G}{p}(Y_{\widetilde{W}}=\vec{1})\,.
\end{multline*}
\end{proof}
Finally we see that for $p\ge\PShearer{G}$ the probability measure $\ShearerMeasure{G}{p}$
\begin{subequations}\label{eq:shearerCharacterizationProbability}
\begin{gather}
	\text{has dependency graph $G$,}
	\label{eq:shearerCharacterizationProbabilityDependencyGraph}\\
	\text{has marginal parameter $p$, i.e $\ForAll v\in V: \ShearerMeasure{G}{p}(Y_v=1)=p$,}
	\label{eq:shearerCharacterizationProbabilityMarginalParameter}\\
	\text{and forbids neighbouring $0$s, i.e. $\ForAll (v,w)\in E: \ShearerMeasure{G}{p}(Y_v=Y_w=0)=0$.}
	\label{eq:shearerCharacterizationProbabilityMarginalNeighbouringZeros}
\end{gather}
\end{subequations}

Every probability measure $\nu$ on $\Set{0,1}^V$ fulfilling \eqref{eq:shearerCharacterizationProbability} can be constructed by \eqref{eq.shearersMeasure} and thus coincides with $\ShearerMeasure{G}{p}$. Hence \eqref{eq:shearerCharacterizationProbability} \emph{characterizes} $\ShearerMeasure{G}{p}$.\\

If $G$ is an \emph{infinite graph} we define
\begin{equation}\label{eq:pShearerInfinite}
	\PShearer{G}:=\sup\Set{\PShearer{H}: H\text{ finite subgraph of }G}\,.
\end{equation}
This is well-defined as $\PShearer{(.)}$ is a monotone increasing function over the lattice of finite subgraphs (strictly monotone increasing for connected subgraphs) \cite[proposition 2.15]{ScottSokal_repulsive}. For $p\ge\PShearer{G}$ the family $\Set{\ShearerMeasure{\Graph{W}}{p}:W\subseteq V,W\text{ finite}}$ forms a consistent family à la \NameKolmogorov{} \cite[(36.1) \& (36.2)]{Billingsley_probability}. Hence \NameKolmogorov{}'s existence theorem \cite[theorem 36.2]{Billingsley_probability} establishes the existence of an extension of this family, which we call $\ShearerMeasure{G}{p}$. The uniqueness of this extension is given by the $\pi$-$\lambda$ theorem \cite[theorem 3.3]{Billingsley_probability}. Furthermore $\ShearerMeasure{G}{p}$ fulfils \eqref{eq:shearerCharacterizationProbability} on the infinite graph $G$. Conversely let $\nu$ be a probability measure having the properties \eqref{eq:shearerCharacterizationProbability}. Then all its finite marginals have them, too and they coincide with \NameShearersMeasure{}. Hence by the uniqueness of the \NameKolmogorov{} extension $\nu$ coincides with $\ShearerMeasure{G}{p}$ and \eqref{eq:shearerCharacterizationProbability} characterizes $\ShearerMeasure{G}{p}$. It follows that the minimal $p$ for \eqref{eq:shearerCharacterizationProbability} to have a solution is $\PShearer{G}$.\\

The reader can find more about \NameShearersMeasure{} in the seminal work by \NameScottSokal{} \cite{ScottSokal_repulsive}, especially the rich connection with \emph{hard core lattice gases} in statistical mechanics and the \emph{\NameLovaszLocalLemma{}} of the probabilistic method in graph theory \cite{ErdosLovasz_hypergraphs}.\\
\subsection{On the \texorpdfstring{$k$-fuzz of $\IntNum$}{k-fuzz of the integers}}
\label{sec:shearerKFuzzZ}
In this section we deal with \NameShearersMeasure{} on $\FuzzKZ$, the \emph{$k$-fuzz} of $\IntNum$. It is the graph with vertices $\IntNum$ and edges for every pair of integers at distance less than or equal to $k$. Recall that an $X$-valued process indexed by $\IntNum$ is called a \emph{$(k+1)$-factor} iff there exists a measurable function $f:[0,1]^{(k+1)}\to X$ such that for every $n\in\IntNum: X_n=f(U_n,\dotsc,U_{n+k})$, where $\Set{U_n}_{n\in\IntNum}$ is a
\Iid{} sequence of Uniform($[0,1]$) rvs. It follows that every $(k+1)$-factor is $k$-independent, stationary and has $\FuzzKZ$ as dependency graph.\\

We derive the critical value $\PShearerKZ$ in proposition \ref{prop:pShearerKFuzzZ} (thus validating \eqref{eq:pShearerKFuzzZDefinition}), construct $\ShearerMeasure{\FuzzKZ}{p}$ explicitly in model \ref{mod:shearerKFuzzZ} as a $(k+1)$-factor and derive asymptotic properties in proposition \ref{prop:shearerKFuzzZEstimates}. For $k\in\NatNumZero$ define the function
\begin{equation}\label{eq:explicitHK}
	\ExplicitHK:\quad [0,1]\to\RealNum \qquad z \mapsto z^k(1-z)\,.
\end{equation}
It attains its maximum at $\frac{k}{k+1}$ with value $\PowerFracDualK$. If $p\in[\PShearerKZ,1]$, then the equation
\begin{equation}\label{eq:theXi}
	\ExplicitHK(\TheXi)=q
\end{equation}
has a \emph{unique solution} $\TheXi:=\TheXi(p,k)$ lying in the interval $[k/(k+1),1]$. Denote by $\FuzzKL{N}$ the $k$-fuzz of a line of $N$ points and by $\FuzzKN$ the $k$-fuzz of $\NatNum$. It is easy to see that $\PShearerKN=\PShearerKZ$ and $\ShearerMeasure{\FuzzKN}{p}$ is just the projection of $\ShearerMeasure{\FuzzKZ}{p}$. Hence all the properties of and estimates for $\ShearerMeasure{\FuzzKZ}{p}$ stated in the following also hold for $\ShearerMeasure{\FuzzKN}{p}$.

\begin{Prop}\label{prop:pShearerKFuzzZ}
\begin{equation}\label{eq:pShearerKFuzzZ}
	\PShearerKL{N}
	\xrightarrow[N\to\infty]{} 1-\PowerFracDualK
	= \PShearerKZ = \PShearerKN\,.
\end{equation}
\end{Prop}

An explicit construction of \NameShearersMeasure{} on $\FuzzKZ$ is given by:

\begin{Mod}\label{mod:shearerKFuzzZ}
Let $p\ge\PShearerKZ$ and $X:=\Set{X_n}_{n\in\IntNum}$ be \Iid{} \NameBernoulli{} rvs with parameter $\TheXi$ as in \eqref{eq:theXi}. Define $Z:=\Set{Z_n}_{n\in\IntNum}$ by
\begin{equation}\label{eq:shearerKFuzzZConstruction}
	\ForAll n\in\IntNum:\quad
	Z_n := 1 - (1-X_n)\prod_{i=1}^k X_{n-i}\,,
\end{equation}
then $Z$ is $\ShearerMeasure{\FuzzKZ}{p}$-distributed.
\end{Mod}
If $k=0$, then the empty product in \eqref{eq:shearerKFuzzZConstruction} disappears and $Z=X$, that is $\ShearerMeasure{\FuzzZeroZ}{p}$ is a \NameBernoulli{} product measure with parameter $p$. Accordingly $\PShearer{\FuzzZeroZ}=0$.\\

A result of \NameAaronsonGilatKeaneDeValk{} \cite[result 4(i) on page 140]{AaronsonGilatKeaneDeValk_algebraic} on the question of the  representability of certain stationary $1$-independent $\Set{0,1}$-valued processes on $\IntNum$ as $2$-factors implies that $\ShearerMeasure{\IntegerSet{n}_1}{p}$ is not representable as a $2$-factor for $p<\frac{3}{4}$. This statement is easily extended to assert non-representability of $\ShearerMeasure{\FuzzKL{n}}{p}$ as a $(k+1)$-factor for every $k,n\in\NatNum$ and $p<\PShearerKZ$. It follows from the fact that for $p<\PShearerKZ$ the sequence $(\beta_n)_{n\in\NatNum}$ in the proof of proposition \ref{prop:pShearerKFuzzZ} does not remain positive.\\

On the other hand, if one fixes $N$ and $p\in[\PShearerKL{N},\PShearerKZ[$, one can get something close to a factor representation. Let $(X_n)_{n=1}^N$ be a collection of independent rvs, with $X_n$ Bernoulli($\beta_n$)-distributed. Then the same rule as in \eqref{eq:shearerKFuzzZConstruction}, truncated for the first $k$ indices, yields a $\ShearerMeasure{\FuzzKL{N}}{p}$-distributed $(Z_n)_{n=1}^N$.\\

\begin{figure}[!htbp]
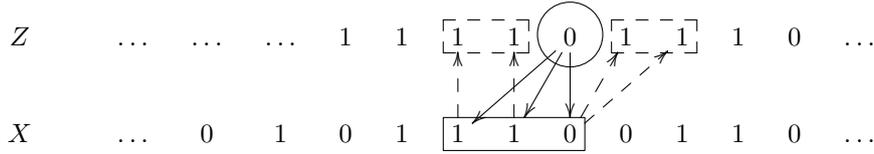

\begin{center}
\begin{displaymath}
\MxyZeroOneSwitch{1em}
\end{displaymath}
\end{center}
\caption[\NameShearersMeasure{} on the $2$-fuzz of $\IntNum$]{A partial view of \NameShearersMeasure{} on the $2$-fuzz of $\IntNum$. The lower row shows a realization of $X$, the upper row the resulting one of $Z$. We point out a $0$ in $Z$, the realizations on its underlying nodes in $X$ (solid downward arrows) and the effect of the zero-one switch (dashed upward arrows), resulting in $1$s on its neighbours up to distance $2$.}
\label{fig:shearer2FuzzZ}
\end{figure}
\begin{proof}[Proof of proposition \ref{prop:pShearerKFuzzZ}]
The inequality $\PShearerKZ\ge 1-\PowerFracDualK$ follows from \cite[theorem 2.1]{LiggettSchonmannStacey_domination}. We repeat it for completeness. Define
 $b_n := \ShearerCriticalFunction{\FuzzKL{n}}(p)$, then $b_n = 1-nq$ for $n\in\IntegerSet{k+1}$ and $b_n=b_{n-1}-q\,b_{n-k}$ for $n>k$, both times using the fundamental identity \eqref{eq:criticalShearerFunctionFundamentalIdentity}. We show by induction that $\beta_n:=\frac{b_n}{b_{n-1}}$ is a strictly monotone falling sequence:\\
$n\in\IntegerSet{k+1}$: $\beta_n=\frac{1-nq}{1-(n-1)q}>\frac{1-(n+1)q}{1-nq}=\beta_{n+1}$ as $n^2>(n-1)(n+1)$.\\
$n\to n+1$: $\beta_{n-k}>\beta_n \Iff  \frac{b_{n-k}}{b_{n}}>\frac{b_{n-k-1}}{b_{n-1}}$ by the induction hypothesis, hence
\begin{equation*}
	\beta_{n+1}=1- \frac{b_{n-k}}{b_{n}}<1-\frac{b_{n-k-1}}{b_{n-1}}=\beta_n\,.
\end{equation*}

The sequence $(\beta_n)_{n\in\NatNum}$ is positive and well-defined iff $p\ge\PShearerKZ$. Upon taking the limit $\beta=\lim_{n\to\infty} \beta_n$ we arrive at the identity $\beta=1-q\beta^{-k}$. Rewrite it to $q=\beta^k(1-\beta)=\ExplicitHK(\beta)$, which has solutions only for $q\le\PowerFracDualK$. Hence $1-\PowerFracDualK\le\PShearerKZ$.\\

The second inequality $\PShearerKZ\le 1-\PowerFracDualK$ follows from model \ref{mod:shearerKFuzzZ}.
\end{proof}
\begin{proof}[Proof of model \ref{mod:shearerKFuzzZ}]
By construction $\Proba(Z_n=0)=\TheXi^k(1-\TheXi)=\ExplicitHK(\TheXi)=q$ and
\begin{align*}
\Proba(Z_n=0)
	&=\Proba(X_{n-1}=\dotso=X_{n-k}=1,X_n=0)\\
	&=\Proba(Z_{n-k}=\dotso=Z_{n-1}=1,Z_n=0,Z_{n+1}=\dotso=Z_{n+k}=1)\,.
\end{align*}
This \emph{zero-one switch} (see figure \ref{fig:shearer2FuzzZ}) guarantees that vertices with distance less than or equal to $k$ can never index a $0$ in the same realization. Therefore $Z$ has no realizations containing neighbouring $0$s with respect to $\FuzzKZ$ as well as the right dependency graph and marginals. Using the characterization \eqref{eq:shearerCharacterizationProbability} we see that $Z$ is $\ShearerMeasure{\FuzzKZ}{p}$-distributed.
\end{proof}

For $k\in\NatNumZero$ fixed define the strictly monotone decreasing function
\begin{equation}\label{eq:theXiFraction}
	\ExplicitFK:
	\quad\Set{0,\dotsc,k}\to\RealNum
	\qquad g\mapsto\begin{cases}
		\frac{(g+1)\TheXi-g}{g\TheXi-(g-1)}&\text{if }k\ge 1\,,\\
		\TheXi&\text{if }k=0\,.
	\end{cases}
\end{equation}

\begin{Prop}\label{prop:shearerKFuzzZEstimates}
We have for every $k\in\NatNumZero$ and $p\ge\PShearerKZ$ the minoration
\begin{subequations}\label{eq:shearerKFuzzZEstimates}
\begin{equation}\label{eq:shearerKFuzzZOVOEPMinoration}
	\ForAll\text{ finite }B\subseteq\IntNum\setminus\Set{0}:\qquad
	\ShearerMeasure{\FuzzKZ}{p}(Y_0=1|Y_B=\vec{1})\ge\ExplicitFK(g_B)\,,
\end{equation}
where $g_B:=0\lor(k+1-d_B)$ and $d_B:=\min\Set{\Modulus{n}:n\in B}$. In particular we have for each $n\in\NatNum$:
\begin{equation}\label{eq:shearerKFuzzZEscapingOVOEPMinoration}
	\ShearerMeasure{\FuzzKZ}{p}(Y_n=1|Y_{\IntegerSet{n-1}}=\vec{1})
	\ge\TheXi
	\quad\text{ and }\quad
	\ShearerCriticalFunction{\FuzzKL{n}}(p)
	\ge\TheXi^n\,.
\end{equation}
And likewise the majoration
\begin{equation}\label{eq:shearerKFuzzZCriticalFunctionMajoration}
	\ForAll \varepsilon>0: \Exists C>0,
	\Exists N\in\NatNum: \ForAll n\ge N:\qquad
	\ShearerCriticalFunction{\FuzzKL{n}}(p)
	\le C [(1+\varepsilon)\TheXi]^n\,.
\end{equation}
\end{subequations}
\end{Prop}

\begin{Rem}
The minimality of \NameShearersMeasure{} \eqref{eq:shearerMinimality} implies that these lower bounds also hold for every $k$-independent \NameBernoulli{} random field on $\IntNum$ and $\NatNum$ with marginal parameter $p\ge\PShearerKZ$ respectively.
\end{Rem}

\begin{NoteToSelf}
PM classifies this as a kind of large deviation result. Equations \eqref{eq:shearerKFuzzZEstimates} imply that
\begin{equation*}
	\lim_{n\to\infty}
		\frac{\log\ShearerCriticalFunction{\FuzzKL{n}}(p)}{n}
	= \log\TheXi\,.
\end{equation*}
\end{NoteToSelf}

\begin{proof}
Fix $p\ge\PShearerKZ$. In the proof of proposition \ref{prop:pShearerKFuzzZ} we see that $(\beta_n)_{n\in\NatNum}$ is a strictly monotone falling sequence with $\beta_n\xrightarrow[n\to\infty]{}\beta\ge\frac{k}{k+1}$. As $\beta$ fulfils $q=\ExplicitHK(\beta)$ we have $\beta=\TheXi$. Hence $\beta_n\ge\TheXi$, yielding \eqref{eq:shearerKFuzzZEscapingOVOEPMinoration}. The monotonicity of $(\beta_n)_{n\in\NatNum}$ implies that
\begin{equation*}
	\ForAll \varepsilon>0: \Exists N\in\NatNum: \ForAll n\ge N:\qquad
	\beta_n\le (1+\varepsilon)\beta = (1+\varepsilon)\TheXi\,.
\end{equation*}
Hence for $n\ge N$ we have
\begin{equation*}
	\ShearerCriticalFunction{\FuzzKL{n}}(p)
	=\prod_{i=1}^n \beta_i
	\le\prod_{i=N+1}^n \beta_i
	\le (1+\varepsilon)^{n-N}\TheXi^{n-N}
	= \frac{1}{[(1+\varepsilon)\TheXi]^N} [(1+\varepsilon)\TheXi]^n\,.
\end{equation*}
This proves \eqref{eq:shearerKFuzzZCriticalFunctionMajoration} upon setting $C(\varepsilon):=[(1+\varepsilon)\TheXi]^{-N}$.\\

For \eqref{eq:shearerKFuzzZOVOEPMinoration} we differentiate according to the shape of $B$. If $d_B>k$, then $k$-independence implies that $\ShearerMeasure{\FuzzKZ}{p}(Y_0=1|Y_B=\vec{1})=p\ge\TheXi=\ExplicitFK(0)$.\\

If $d_B\le k$ let $B_{\pm}:=B\cap\IntNum_{\pm}$ and $d_{B_\pm}:=\inf\Set{\Modulus{n}:n\in B_\pm}$. Thus $d_B=d_{B_-}\land d_{B_+}$. In the first case $d_{B_-}>k$ and $d_{B_+}\le k$. Let $\Set{b_1,\dotsc,b_m}:=B_{+}\cap\IntegerSet{k}$ with $b_1<\dotso<b_m$.
 Hence
\begin{align*}
	&\FirstAlign1-\ShearerMeasure{\FuzzKZ}{p}(Y_0=1|Y_B=\vec{1})\\
	&=1-\ShearerMeasure{\FuzzKZ}{p}(Y_0=1|Y_{B_{+}}=\vec{1})
	&\text{by $k$-independence}\\
	&= \frac{q}{\prod_{i=1}^m \ShearerMeasure{\FuzzKZ}{p}(Y_{b_i}=1|Y_{B_{+}\setminus\Set{b_1,\dotsc,b_{i-1}}}=\vec{1})}
	&\text{fundamental identity \eqref{eq:criticalShearerFunctionFundamentalIdentity}}\\
	&\le \frac{q}{\TheXi^m}
	&\text{by induction over $\Cardinality{B_{+}}$}\\
	&\le (1-\TheXi)\TheXi^{k-m}
	&\text{as }q=(1-\TheXi)\TheXi^k\\
	&\le 1-\TheXi
	&\text{as }m\le k\,.
\end{align*}
This also holds in the symmetric case with $d_{B_-}\le k$ and $d_{B_+}>k$.\\

The final case is $d_{B_+}\le k$ and $d_{B_-}\le k$. Assume without loss of generality that $d_B=d_{B_-}\le d_{B_+}$ and let $\Set{a_n,\dotsc,a_1}:=B_{-}\cap\Set{-k,\dotsc,-1}$ with $a_n<\dotso<a_1$. Applying the fundamental identity \eqref{eq:criticalShearerFunctionFundamentalIdentity} and induction over $\Cardinality{B}$ we get
\begin{align*}
	&\FirstAlign1-\ShearerMeasure{\FuzzKZ}{p}(Y_0=1|Y_B=\vec{1})\\
	&= \frac{q}{
		\prod_{j=1}^n\ShearerMeasure{\FuzzKZ}{p}(Y_{a_j}=1|
			Y_{B_{-}\setminus\Set{a_1,\dotsc,a_{j-1}}}=\vec{1},
			Y_{B_{+}}=\vec{1})}\\
	&\quad\times\frac{1}{
		\prod_{i=1}^m \ShearerMeasure{\FuzzKZ}{p}(Y_{b_i}=1|
			Y_{B_{-}\setminus\Set{a_1,\dotsc,a_n}}=\vec{1},
			Y_{B_{+}\setminus\Set{b_1,\dotsc,b_{i-1}}}=\vec{1})}\\
	&\le\frac{q}{
		\prod_{j=1}^n \ExplicitFK(k+a_j)
		\prod_{i=1}^m \ExplicitFK(0)}\\
	&\le\frac{q}{
		\prod_{j=d_B}^k \ExplicitFK(k-j)
		\prod_{i=1}^k \ExplicitFK(0)}\\
	&=\frac{(1-\TheXi)\TheXi^k}{
		\left[\prod_{j=d_B}^k \frac%
			{(k+1-j)\TheXi-(k-j)}
			{(k-j)\TheXi-(k-1-j)}
		\right]
		\TheXi^k}\\
	&=\frac{1-\TheXi}
		{(k+1-d_B)\TheXi-(k-d_B)}\,.
\end{align*}
It follows that
\begin{multline*}
  \ShearerMeasure{\FuzzKZ}{p}(Y_0=1|Y_B=\vec{1})
  \ge 1-\frac{1-\TheXi}
		{(k+1-d_B)\TheXi-(k-d_B)}\\
  = \frac{(k+2-d_B)\TheXi-(k+1-d_B)}
		{(k+1-d_B)\TheXi-(k-d_B)}
  =\ExplicitFK(k+1-d_B)\,.
\end{multline*}
\end{proof}
\section{Proofs}\label{sec:proofs}
\subsection{Proof outline of theorems \ref{thm:valueOfRootedCriticalValues} and \ref{thm:diameterBoundedThreshold}}
\label{sec:mainProofOutlines}
\begin{proof}[Proof of theorem \ref{thm:valueOfRootedCriticalValues}]
We start with some obvious relations between the rooted percolation classes and their critical values, based on the restrictions imposed by $k$ and $s$. For all $k,k',s,s'\in\NatNumZero$:
\begin{equation}\label{eq:rootedPercolationClassRelations}
	\text{if $k\le k'$ and $s\le s'$ then }
	\begin{cases}
		\PercolationClassRooted{p}{o}{k}{s}{V}
			\subseteq\PercolationClassRooted{p}{o}{k'}{s'}{V}\\
		\PMaxV{k,s}\le\PMaxV{k',s'}\\
		\PMinV{k',s'}\le\PMinV{k,s}
	\end{cases}
	\text{ holds.}
\end{equation}

The first part is the proof for $\PMaxV{k,s}$ in sections \ref{sec:upperBoundPMax} and \ref{sec:lowerBoundPMax}. To get an upper bound on $\PMaxV{k,s}$ we need to show that every $k,s$-independent percolation percolates for $p$ close enough to $1$. Our approach uses a classical second moment argument, recalled in lemma \ref{lem:secondMomentMethod}. We relate it to $\BranchingNumberT$ in proposition \ref{prop:secondMomentMethodTreeAdaption}, with the core ingredient being a sufficient condition for percolation in terms of an exponential bound on the percolation kernel, defined in \ref{eq:percolationKernel}. For $k,s$-independent percolation proposition \ref{prop:ksIndependentKernelBound} reduces this to the problem of bounding the conditional probabilities of extending open geodesic downrays by the right exponential term. Finally proposition \ref{prop:secondMomentBound} uses the minimality of \NameShearersMeasure{} from lemma \ref{lem:shearerMinimality} and detailed estimates about its structure on $\FuzzKZ$ in proposition \ref{prop:shearerKFuzzZEstimates} to uniformly guarantee the right exponential term and arrive at \eqref{eq:secondMomentBound}:
\begin{equation*}
	\ForAll k,s\in\NatNumZero:\qquad
	\PMaxV{k,s}\le\begin{cases}
		\ExplicitGK(\BranchingNumberT)
		&\text{if }\BranchingNumberT\le\frac{k+1}{k}\\
		\PShearerKZ
		&\text{if }\BranchingNumberT\ge\frac{k+1}{k}\,.
	\end{cases}
\end{equation*}
For the lower bound on $\PMaxV{k,s}$ it suffices to exhibit $k,0$-independent percolation models that do not percolate. We describe two such models, the canonical model \ref{mod:canonical} and the cutup model \ref{mod:cutup}, both constructed from \NameShearersMeasure{}. More precisely, in section \ref{sec:treeFission} we describe a general procedure, called tree-fission, to create a $k,0$-independent percolation with identical distributions along all downrays from a given $k$-independent \NameBernoulli{} random field over $\NatNum$. When applied to \NameShearersMeasure{} on $\FuzzKN$ and a derivative of $\FuzzKL{n}$ it yields the canonical model \ref{mod:canonical} and the cutup model \ref{mod:cutup} respectively. We then use the first moment method, recalled in lemma \ref{lem:firstMomentMethod}, to establish their nonpercolation, leading to the following results from \eqref{eq:canonicalModelBound} and  \eqref{eq:cutupModelBound}:
\begin{equation*}
	\PMaxV{k,0}\ge\ExplicitGK(\BranchingNumberT)\quad
	\text{if }\BranchingNumberT\in\left[1,\frac{k+1}{k}\right[
	\quad\text{ and }\quad
	\PMaxV{k,0}\ge\PShearerKZ\,.
\end{equation*}
Conclude by applying the inequality from \eqref{eq:rootedPercolationClassRelations}.\\

The second part is the proof for $\PMinV{k,s}$ in section \ref{sec:determiningPMin}. Here the argumentation is the reverse of the one for $\PMaxV{k,s}$. To get a lower bound on $\PMinV{k,s}$ we need to show that every $k,s$-independent percolation does not percolate for $p$ close enough to $0$. We achieve this by a first moment argument in proposition \ref{prop:firstMomentBound}, using solely $k$-independence along downrays. It culminates in \eqref{eq:firstMomentBound}:
\begin{equation*}
	\ForAll k,s\in\NatNumZero:\qquad
	\PMinV{k,s}\ge\frac{1}{\BranchingNumberT^{k+1}}\,.
\end{equation*}
For the upper bound on $\PMinV{k,s}$ we differentiate between $k=0$ and $k\ge 1$. In the case $k=0$ we already have a matching upper bound in the upper bound for $\PMaxV{0,s}$ in \eqref{eq:secondMomentBound}. For $k\ge 1$ we describe a percolating $k,0$-independent percolation model, called the minimal model \ref{mod:minimal}. It is constructed by the tree-fission procedure from section \ref{sec:treeFission}. In proposition \ref{prop:minimalModelFacts} we show that it percolates by bounding its percolation kernel with the help of proposition \ref{prop:ksIndependentKernelBound} and applying the second moment method adaption from proposition \ref{prop:secondMomentMethodTreeAdaption}, leading to \eqref{eq:minimalModelBound}:
\begin{equation*}
	\ForAll k\ge 1:\qquad
	\PMinV{k,0}\le\frac{1}{\BranchingNumberT^{k+1}}\,.
\end{equation*}
Conclude by applying the inequality from \eqref{eq:rootedPercolationClassRelations}, using the upper bound for $\PMaxV{0,s}$ in the case of $k=0$.
\end{proof}

\begin{proof}[Proof of theorem \ref{thm:diameterBoundedThreshold}]
By \eqref{eq:pShearerKFuzzZ} for every $\varepsilon>0$ there exists a $N\in\NatNum$ such that $\PShearerKZ>\PShearerKL{N}>\PShearerKZ-\varepsilon$. Then proposition \ref{prop:cutupModelFacts} asserts that the cutup percolation $\PercolationCutNK$ (model \ref{mod:cutup}) is diameter bounded with $D=4N-4$.\\

On the other hand, let $p\ge\PShearerKZ$ and $Z:=\Set{Z_v}_{v\in V}$ be in $\PercolationClassRooted{p}{o}{k}{s}{V}$. We have
\begin{equation*}
  \ForAll n\in\NatNum, v\in\TreeLevel{\Tree}{n}:\quad
  \Proba(Z_{\TreePath{o}{v}}= \vec{1})
  \ge \ShearerMeasure{\FuzzKZ}{p}(Y_{\IntegerSet{n}}=\vec{1})
  \ge \TheXi^n > 0\,,
\end{equation*}
where $Y$ is $\ShearerMeasure{\FuzzKZ}{p}$-distributed, we use the minimality of \NameShearersMeasure{} \eqref{eq:shearerMinimalityCriticalFunction} and the minoration from \eqref{eq:shearerKFuzzZEscapingOVOEPMinoration}, with $\TheXi>0$ from \ref{eq:theXi}. This implies that $Z$ is not diameter bounded.
\end{proof}
\subsection{General tools for percolation on trees}
\label{sec:generalTreePercolationTools}
In this section we list some general tools for percolations on trees which allow us to shorten the following proofs. The following \emph{extension} of \emph{\NameKolmogorov{}'s zero-one law} \cite[theorem 36.2]{Billingsley_probability} is well known. In particular it encompasses $k$-independent rvs on a graph $G$, as they have the $k$-fuzz of $G$ as their dependency graph.
\begin{Lem}\label{lem:kolmogorovZeroOneLocallyDependent}
Let $G=(V,E)$ be a locally finite, infinite graph. Let $X:=\Set{X_v}_{v\in V}$ be a random field with dependency graph $G$. Then the tail $\sigma$-algebra of $X$ is trivial.
\end{Lem}
\begin{proof}
Let $(V_n)_{n\in\NatNum}$ be an exhausting, strictly monotone growing sequence of finite subsets of $V$. For $W\subseteq V$ let $\Algebra_W:=\sigma(X_W)$ and define the tail $\sigma$-algebra $\Algebra_\infty:=\bigcap_{n=1}^\infty \Algebra_{V_n^c}$. For an event $B\in\Algebra_\infty$ set $Z_n:=\Expect[\Indicator{B}|\Algebra_{V_n}]=\Indicator{B}$. Then we have a \AlmostSureLy{} constant martingale with $\lim_{n\to\infty} Z_n=\Indicator{B}$:
\begin{equation*}
  \Expect[Z_{n+1}|\Algebra_{V_n}]
= \Expect[\Expect[\Indicator{B}|\Algebra_{V_{n+1}}]|\Algebra_{V_n}]
= \Expect[\Expect[\Indicator{B}|\Algebra_{V_n}]|\Algebra_{V_{n+1}}]
= \Expect[\Indicator{B}|\Algebra_{V_n}]
= Z_n\,.
\end{equation*}
Hence $\Proba(B)^2=\Expect[\Indicator{B}\Proba(B)]=\Expect[\Indicator{B}^2]=\Proba(B)$ and $\Proba(B)\in\Set{0,1}$.
\end{proof}
Next we introduce some notation for rooted percolation on $\Tree$:
\begin{Not}\label{not:openTo}
In the context of rooted percolation and for $v\in V$ we write
\begin{subequations}\label{eq:openTo}
\begin{align}
	\label{eq:openToCutset}
	O_v^\Pi
	&:= \Set{v\ConnectedTo \Pi\cap\Vertices{\SubtreeRootedAt{\Tree}{v}}}
	&\Pi\in\CutsetsOf{o}\\
	\label{eq:openToBoundary}
	O_v
	&:= \Set{v\ConnectedTo\infty}
	= \Set{v\ConnectedTo\TreeBoundary{\SubtreeRootedAt{\Tree}{v}}}\,,
\end{align}
\end{subequations}
where those events mean ``there is an \emph{open downpath} from $w$ to the cutset $\Pi$'' and ``there is an \emph{open downray} starting at $v$''.
\end{Not}
The following lemma allows us to concentrate exclusively on rooted percolation (see \cite{Temmel_kindependent} for a proof):
\begin{Lem}\label{lem:percolationRootedVsUnrooted}
Let $\Percolation\in\PercolationClass{p}{k}{V}$, for finite $k$. Then
\begin{subequations}
\begin{align}
	(\Exists v\in V: \Proba(O_v) > 0) \Iff \Proba(\Percolation\text{ percolates on $\Tree$}) = 1\,,
		\label{eq:percolationRootedVsUnrootedPositive}\\
	(\ForAll v\in V: \Proba(O_v) = 0) \Iff \Proba(\Percolation\text{ percolates on $\Tree$}) = 0\,.
		\label{eq:percolationRootedVsUnrootedZero}
\end{align}
\end{subequations}
\end{Lem}
In the case $k=s=0$ we can change the $\Exists$ to $\ForAll$ in \eqref{eq:percolationRootedVsUnrootedPositive}, which is needed in the proof of proposition \ref{prop:minimalModelFacts}. Finally the obvious relationship between rooted percolation reaching a cutset $\Pi\in\CutsetsOf{o}$ or the boundary $\TreeBoundary{\Tree}$ from $o$ is:
\begin{equation}\label{eq:openToCutsetIntersectionEqualsBoundary}
	\ForAll w\in V:\quad
	O_w = \bigcap_{\Pi\in\CutsetsOf{o}} O_w^\Pi\,.
\end{equation}
This holds already for the intersection over an \emph{exhaustive sequence of cutsets} $\Set{\Pi_m}_{m\in\NatNum}$, i.e. $\ForAll v\in V: \Exists m_v\in\NatNum: \Exists w\in\Pi_{m_v}: v$ is an ancestor of $w$. A central tool is the following two moment methods:
\begin{Lem}[First moment method {\cite[section 5.2]{LyonsPeres_prob_tree_network}}]
\label{lem:firstMomentMethod}
We have
\begin{equation}\label{eq:firstMomentMethod}
	\Proba(O_o)
= \Proba(o\ConnectedTo\infty)
\le\inf_{\Pi\in\CutsetsOf{o}} \sum_{v\in\Pi} \Proba(o\ConnectedTo v)\,.
\end{equation}
\end{Lem}
\begin{Lem}[Weighted second moment method {\cite[section 5.3]{LyonsPeres_prob_tree_network}}]
\label{lem:secondMomentMethod}
\begin{subequations}
\begin{equation}\label{eq:secondMomentMethod}
	\Proba(O_o)
= \Proba(o\ConnectedTo\infty)
\ge \inf_{\Pi\in\CutsetsOf{o}} \sup_{\mu\in\ProbabilityMeasureSpace{\Pi}} \frac{1}{\Energy{\mu}}\,,
\end{equation}
where $\ProbabilityMeasureSpace{\Pi}$ is the set of probability measures on the vertex cutset $\Pi$ and the \emph{energy} $\Energy{\mu}$ of $\mu\in\ProbabilityMeasureSpace{\Pi}$ is determined by
\begin{equation}\label{eq:muEnergy}
	\Energy{\mu} = \sum_{v,w\in\Pi} \mu(v)\mu(w)\PercolationKernel(v,w)\,.
\end{equation}
and $\PercolationKernel$ is the symmetric \emph{percolation kernel}
\begin{equation}\label{eq:percolationKernel}
	\PercolationKernel:
		\qquad V^2 \to \RealNumPlus
		\qquad (v,w)\mapsto\PercolationKernel(v,w)
		:=\frac%
			{\Proba(o\ConnectedTo v,o\ConnectedTo w)}
			{\Proba(o\ConnectedTo v)\Proba(o\ConnectedTo w)}\,.
\end{equation}
\end{subequations}
\end{Lem}
\subsection{Upper bound on \texorpdfstring{$\PMaxV{k,s}$}{the maximal p}}
\label{sec:upperBoundPMax}
The task is to establish an upper bound on $\PMaxV{k,s}$. In other words, we want to guarantee percolation for high enough $p$. The first step in section \ref{sec:percolationKernelEstimates} is to use the second moment method to translate this problem into the search for a suitable exponential bound on the percolation kernel. Then we use $k,s$-independence to bound the percolation kernel in terms of a conditional probability along a single downray. Hence we can guarantee percolation as soon as we can bound this conditional probability from below in sufficient exponential terms. The percolation along a single downray is just a \NameBernoulli{} random field with parameter $p$ and dependency graph $\FuzzKN$. In the second step in section \ref{sec:uniformBoundByShearerMeasure} we apply the generic minimality of \NameShearersMeasure{} and a lower bound on $\ShearerMeasure{\FuzzKN}{p}$ to get such an exponential lower bound of parameter $\TheXi$. Finally we relate $\TheXi$ and $\BranchingNumberT$ and derive the upper bound.
\subsubsection{Percolation kernel estimates}
\label{sec:percolationKernelEstimates}
In proposition \ref{prop:secondMomentMethodTreeAdaption} we state a sufficient condition on the percolation kernel in order to percolate. This condition relates the second moment method to the branching number. In proposition \ref{prop:ksIndependentKernelBound} we bound the percolation kernel for $k,s$-independent percolation in terms of conditional probabilities along a single downray, hence providing a simpler means to derive the sufficient condition in subsequent steps.
\begin{Prop}\label{prop:secondMomentMethodTreeAdaption}
Let $\Percolation\in\PercolationClassRooted{?}{o}{k}{s}{V}$ and $\alpha<\BranchingNumberT$, $C\in\RealNumPlus$ such that $\ForAll v,w\in V$:
\begin{equation}\label{eq:secondMomentMethodTreeAdaption}
	\PercolationKernel(v,w)\le C \alpha^{\NodeLevel{v\Confluent w}}\,,
\end{equation}
then $\Percolation$ percolates.
\end{Prop}
\begin{Rem}
The ``?'' in $\PercolationClassRooted{?}{o}{k}{s}{V}$ means that we place no restriction yet on the marginals of $\Percolation$. The \emph{confluent} of $v$ and $w$ is $v\Confluent w$. See also figure \ref{fig:percolationKernel}.
\end{Rem}
\begin{proof}
Take $\beta\in]\alpha,\BranchingNumberT[$ and let $g$ be a $\beta$-flow. Define $\mu(v):=\frac{g(v)}{g(o)}$, hence $\mu|_\Pi\in\ProbabilityMeasureSpace{\Pi}$ for each vertex cutset $\Pi\in\CutsetsOf{o}$. We have

\begin{align*}
	&\FirstAlign\Energy{\mu|_\Pi}\\
	&=\sum_{v,w\in\Pi} \mu|_\Pi(v)\mu|_\Pi(w)\PercolationKernel(v,w)\\
	&\le \sum_{v,w\in\Pi} \mu(v)\mu(w)
		C \alpha^{\NodeLevel{v\Confluent w}}\\
	&=C
		\sum_{n=0}^{\infty} \alpha^n
		\sum_{\substack{
			v,w\in\Pi\\
			v\Confluent w=:u\in\TreeLevel{\Tree}{n}
		}}
		\mu(v)\mu(w)\\
	&\le C \sum_{n=0}^{\infty} \alpha^n
		\sum_{u\in\TreeLevel{\Tree}{n}}
		\sum_{\substack{
			v,w\in\Pi\\
			u\in\TreePath{o}{v\Confluent w}
		}}
		\frac{g(v)g(w)}{g(o)^2}
		&\text{more nodes}\\
	&= \frac{C}{g(o)^2}
		\sum_{n=0}^{\infty} \alpha^n
		\sum_{u\in\TreeLevel{\Tree}{n}} g(u)^2
		&\text{flow property}\\
	&\le \frac{C}{g(o)^2}
		\sum_{n=0}^{\infty} \left(\frac{\alpha}{\beta}\right)^n
		\sum_{u\in\TreeLevel{\Tree}{n}} g(u)
		&\beta\text{-flow}\\
	&\le \frac{C}{g(o)}
		\sum_{n=0}^{\infty} \left(\frac{\alpha}{\beta}\right)^n
		&\text{flow property}\\
	&= \frac{C}{g(o)} \frac{\beta}{\beta-\alpha}
		&\alpha<\beta,
\end{align*}
which is a finite bound independent of $\Pi$. Apply the weighted second moment method (see lemma \ref{lem:secondMomentMethod}) to see that $\Proba(o\ConnectedTo\infty)>0$ and conclude.
\end{proof}
\begin{figure}[!htbp]
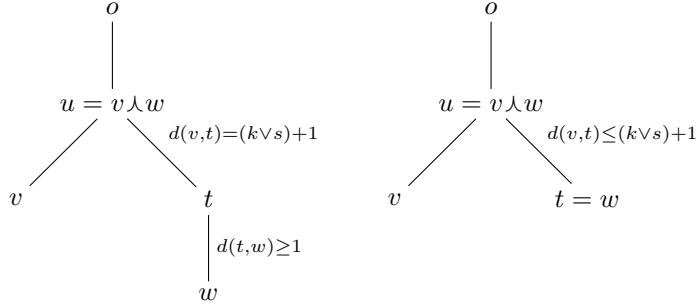

\begin{center}
\begin{displaymath}
\MxyPercolationKernelNormal\qquad
\MxyPercolationKernelDegenerated
\end{displaymath}
\end{center}
\caption[Percolation kernel decomposition]{Decomposition of the percolation kernel $\PercolationKernel(v,w)$ for $k,s$-independent, rooted site percolation. The node $t\in\TreePath{u}{w}$ has distance $(k\lor s)+1$ from $u$ if the path $\TreePath{u}{w}$ is longer than this (left side), otherwise $t=w$ (right side).}
\label{fig:percolationKernel}
\end{figure}
\begin{Prop}\label{prop:ksIndependentKernelBound}
We use the notation from figure \ref{fig:percolationKernel}. Then $\ForAll k,s\in\NatNumZero,
	\Percolation\in\PercolationClassRooted{p}{o}{k}{s}{V}, v,w\in V$:
\begin{equation}\label{eq:ksIndependentKernelBound}
	\PercolationKernel(v,w)\le\frac{1}{\Proba(o\ConnectedTo t|t\ConnectedTo w)}\,.
\end{equation}
\end{Prop}

\begin{proof}
We use the notation from figure \ref{fig:percolationKernel}. In the case $\GraphMetric{u}{w}>(k\lor s)+1$ we have
\begin{align*}
	&\FirstAlign\PercolationKernel(v,w)\\
	&= \frac%
		{\Proba(o\ConnectedTo v,o\ConnectedTo w)}
		{\Proba(o\ConnectedTo v)\Proba(o\ConnectedTo w)}\\
	&= \frac%
		{\Proba(o\ConnectedTo v,o\ConnectedTo w)}
		{\Proba(o\ConnectedTo v)\Proba(t\ConnectedTo w)\Proba(o\ConnectedTo t|t\ConnectedTo w)}\\
	&= \frac%
		{\Proba(o\ConnectedTo v,u\ConnectedTo w)}
		{\Proba(o\ConnectedTo v,t\ConnectedTo w)}
		\frac{1}{\Proba(o\ConnectedTo t|t\ConnectedTo w)}
		&\text{using }(k\lor s)\text{-independence}\\
	&\le \frac{1}{\Proba(o\ConnectedTo t|t\ConnectedTo w)}\ .
\end{align*}

In the case $\GraphMetric{u}{w}\le(k\lor s)+1$ we have $t=w$ and
\begin{equation*}
	\FirstAlign\PercolationKernel(v,w)
	= \frac%
		{\Proba(o\ConnectedTo v,o\ConnectedTo w)}
		{\Proba(o\ConnectedTo v)\Proba(o\ConnectedTo w)}
	= \Proba(u\ConnectedTo t|o\ConnectedTo v)\frac{1}{\Proba(o\ConnectedTo t)}
	\le \frac{1}{\Proba(o\ConnectedTo t|t\ConnectedTo w)}\ .
\end{equation*}
\end{proof}
\subsubsection{Uniform bound by \NameShearersMeasure{}}
\label{sec:uniformBoundByShearerMeasure}
The following proposition combines our knowledge of $\ShearerMeasure{\FuzzKZ}{p}$ and its properties with the simplified condition on the percolation kernel from proposition \ref{prop:ksIndependentKernelBound} to ensure uniform percolation.

\begin{Prop}\label{prop:secondMomentBound}
\begin{equation}\label{eq:secondMomentBound}
\ForAll k,s\in\NatNumZero:\quad
	\PMaxV{k,s}\le\begin{cases}
		\ExplicitGK(\BranchingNumberT)
		&\text{if }\BranchingNumberT\le\frac{k+1}{k}\\
		\PShearerKZ
		&\text{if }\BranchingNumberT\ge\frac{k+1}{k}\,.
	\end{cases}
\end{equation}
Furthermore for $\BranchingNumberT>\frac{k+1}{k}$ every percolation in $\PercolationClassRooted{\PShearerKZ}{o}{k}{s}{V}$ percolates. In the case $k=0$ we interpret $\frac{1}{0}:=\infty$.
\end{Prop}

\begin{proof}
Let $p\ge\PShearerKZ$. Use the notation from figure \ref{fig:percolationKernel}. Let $\TheXi$ be the unique solution of the equation $1-p=\TheXi(1-\TheXi)^k$ from \eqref{eq:theXi}. In a first step we use \eqref{eq:ksIndependentKernelBound}, the minimality of \NameShearersMeasure{} \eqref{eq:shearerMinimalityCriticalFunction}, the explicit minoration of \NameShearersMeasure{} on $\FuzzKN$ \eqref{eq:shearerKFuzzZEscapingOVOEPMinoration} and the fact that $\NodeLevel{t}\le\NodeLevel{u}+(k\lor s)+1$ to majorize the percolation kernel as follows:
\begin{equation*}
    \PercolationKernel(v,w)
\le \frac{1}{\Proba(o\ConnectedTo t|t\ConnectedTo w)}
\le \frac{1}{\ShearerMeasure{\FuzzKZ}{p}(o\ConnectedTo t|t\ConnectedTo w)}
\le \frac{1}{\TheXi^{\NodeLevel{t}}}
\le \TheXi^{-(k\lor s)-1}
    \TheXi^{-\NodeLevel{u}}\,.
\end{equation*}
In the second step we want to apply the sufficient exponential bound condition on the percolation kernel from proposition \ref{prop:secondMomentMethodTreeAdaption}, hence we have to relate $\TheXi$ with $\BranchingNumberT$. The function $\ExplicitGK$ \eqref{eq:explicitGK} satisfies $\ExplicitGK(\frac{1}{\TheXi})=p$, has a global minimum in $\frac{k+1}{k}$ with value $\PShearerKZ$ and induces a strictly monotone decreasing bijection between $[1,\frac{k+1}{k}]$ and $[\PShearerKZ,1]$.\\

Case $\BranchingNumberT\le\frac{k+1}{k}$ and $\ExplicitGK(\BranchingNumberT)<p=\TheXi^k(1-\TheXi)$: Apply proposition \ref{prop:secondMomentMethodTreeAdaption} with $C:=\TheXi^{-(k\lor s)-1}$ and $\alpha:=\frac{1}{\TheXi}<\BranchingNumberT$ to show that we percolate. This proves the $\ExplicitGK$ part of \eqref{eq:secondMomentBound}.\\

Case $\BranchingNumberT>\frac{k+1}{k}$ and $\PShearerKZ\le p$: Apply proposition \ref{prop:secondMomentMethodTreeAdaption} with $C:=\TheXi^{-(k\lor s)-1}$ and $\alpha:=\frac{1}{\TheXi}\le\frac{k+1}{k}<\BranchingNumberT$ to show that we percolate. This proves the $\PShearerKZ$ part of \eqref{eq:secondMomentBound} and the percolation statement at $\PShearerKZ$.
\end{proof}

We show that we need uniformly bounded elsewhere-dependences to guarantee percolation for high $p$. The counterexample consists of \emph{multiplexing} a distribution indexed by $\NatNumZero$ over the corresponding level of $\Tree$.

\begin{Mod}\label{mod:multiplex}
For $p\ge\PShearerKZ$ let $\mathcal{Z}:=\Set{\mathcal{Z}_n}_{n\in\NatNumZero}$ be a collection of $k$-independent Bernoulli($p$) rvs. Define a site percolation $Z:=(Z_v)_{v\in V}$ on the rooted tree $\Tree$ by
\begin{equation}\label{eq:unboundedElsewhereDependencyCounterExample}
	Z_v:=\mathcal{Z}_{\NodeLevel{v}}\,.
\end{equation}
\end{Mod}

\begin{Prop}\label{prop:multiplexFacts}
For every $s\in\NatNum$ we have $Z\not\in\PercolationClassRooted{p}{o}{k}{s}{V}$ and $Z$ percolates iff $p=1$.
\end{Prop}

\begin{proof}
All the sites on a chosen level of $\Tree$ realize \AlmostSureLy{} in the same state. Therefore the elsewhere-dependence $s_v$ of $v$ is in the range $2\,\NodeLevel{v}\le s_v\le 2\,\NodeLevel{v}+ k$ and unbounded in $v$. Using \eqref{eq:unboundedElsewhereDependencyCounterExample} and $k$-independence we get
\begin{equation*}
	\ForAll n\in\NatNum:\quad
	\Proba(o\ConnectedTo\TreeLevel{\Tree}{n})
	=\Proba(\mathcal{Z}_0=\dotso=\mathcal{Z}_n=1)
	\le p^{n/(k+1)}\,.
\end{equation*}
This exponential upper bound implies that $\Proba(O_o)=0$ iff $p<1$.
\end{proof}
\subsection{Lower bound on \texorpdfstring{$\PMaxV{k,s}$}{the maximal p}}
\label{sec:lowerBoundPMax}
To derive a lower bound on $\PMaxV{k,s}$ we exhibit appropriate nonpercolating percolation models. The proof of proposition \ref{prop:secondMomentBound} suggests to look for percolations being $\ShearerMeasure{\FuzzKZ}{p}$-distributed along downrays. To be as general as possible we also want $s=0$. Section \ref{sec:treeFission} presents a procedure to construct a $k,0$-independent percolation model with given distribution along downrays. We then apply this construction to probability distributions derived from  $\ShearerMeasure{\FuzzKZ}{p}$ and $\ShearerMeasure{\FuzzKL{N}}{p}$. Applying the first moment method and relating the relevant parameters to $\BranchingNumberT$ yields the lower bounds.
\subsubsection{Tree fission}\label{sec:treeFission}
In this section we show how to create a $k,0$-independent percolation model from a $k$-independent \NameBernoulli{} random field $\mathcal{Z}$ indexed by $\NatNumZero$. Additionally the resulting model has the same distribution along all downrays, namely the one of $\mathcal{Z}$, and is invariant under automorphisms of the rooted tree. The generic construction is presented in proposition \ref{prop:treeFission} and specialized to our setting in corollary \ref{cor:treeFissionInRootedClass}.

\begin{Prop}\label{prop:treeFission}
Let $\mathcal{Z}:=\Set{\mathcal{Z}_n}_{n\in\NatNumZero}$ be a \NameBernoulli{} random field and $\Tree:=(V,E)$ be a tree rooted at $o$. Then there exists a unique probability measure $\nu$, called the \emph{$\Tree$-fission of $\mathcal{Z}$}, under which the \NameBernoulli{} field $Z:=\Set{Z_v}_{v\in V}$ has the following properties:
\begin{subequations}\label{eq:treeFission}
\begin{multline}\label{eq:treeFissionSubtreeIndependence}
	\ForAll W\subseteq V:\quad
		\text{if }\ForAll v,w\in W:
			v\not\in\Vertices{\SubtreeRootedAt{\Tree}{w}}\,,\\
	\text{then the subfields }
	\Set{Z_{\Vertices{\SubtreeRootedAt{\Tree}{w}}}}_{w\in W}
	\text{ are independent.}
\end{multline}
\begin{equation}\label{eq:treeFissionPathMarginal}
	\ForAll v\in V:\quad
		Z_{\TreePath{o}{v}}\text{ has the same law as }
		\Set{\mathcal{Z}_{\NodeLevel{w}}}_{w\in\TreePath{o}{v}}\,.
\end{equation}
\end{subequations}
Furthermore $Z$ is invariant under automorphisms of the rooted tree.
\end{Prop}

\begin{proof}
For $v\in V$ let $A(v):=\TreePath{o}{v}\setminus\Set{v}$ be the set of all ancestors of $v$. Let $\mathcal{S}$ be the family of vertices of finite connected components of $V$ containing $o$. For $R\in\mathcal{S}$ define the probability measure $\nu_R$ on $\Set{0,1}^R$ by setting
\begin{equation}\label{eq:treeFissionDefinition}
	\ForAll \vec{s}_R\in\Set{0,1}^R:\quad
		\nu_R(Y_R=\vec{s}_R)
		:= \prod_{v\in R}
			\Proba(\mathcal{Z}_{\NodeLevel{v}}=s_v|
				\ForAll w\in A(v):\mathcal{Z}_{\NodeLevel{w}}=s_w)\,.
\end{equation}

We claim that $\Set{\nu_R}_{R\in\mathcal{S}}$ is a consistent family à la \NameKolmogorov{}. Furthermore each $\nu_R$ has properties \eqref{eq:treeFission}. One can prove these claims by induction over the cardinality of $R$ (omitted). Hence \NameKolmogorov{}'s existence theorem \cite[theorem 36.2]{Billingsley_probability} yields an extension $\nu$ of the above family. The probability measure $\nu$ fulfils \eqref{eq:treeFission} because all its marginals $\nu_R$ do so. Uniqueness follows from the fact that the properties \eqref{eq:treeFission} imply the construction of the marginal laws $\nu_R$ via \eqref{eq:treeFissionDefinition} and the $\pi-\lambda$ theorem \cite[theorem 3.3]{Billingsley_probability}.
\end{proof}

\begin{Cor}\label{cor:treeFissionInRootedClass}
If $\mathcal{Z}$ from proposition \ref{prop:treeFission} is $k$-independent and has marginal parameter $p$ then $\nu$, the $\Tree$-fission of $\mathcal{Z}$, is the law of a percolation in $\PercolationClassRooted{p}{o}{k}{0}{V}$ invariant under automorphisms of the rooted tree.
\end{Cor}

\begin{proof}
The definition of $\nu$ implies that it is the law of a rooted site percolation which is invariant under automorphisms of the rooted tree. $k$-independence and the fact that $\nu(Y_v=1)=p$ follow from \eqref{eq:treeFissionPathMarginal}, while $s=0$ follows from the independence over disjoint subtrees in \eqref{eq:treeFissionSubtreeIndependence}.
\end{proof}
\subsubsection{The canonical model}\label{sec:canonical}
For $p\ge\PShearerKZ$ we derive a $k,0$-independent percolation model from $\ShearerMeasure{\FuzzKZ}{p}$. It does not percolate for small $\BranchingNumberT$ if $p$ is smaller than the $\ExplicitGK$ part of \eqref{eq:secondMomentBound}, leading to a lower bound on $\PMaxV{k,0}$.

\begin{Mod}\label{mod:canonical}
Let $k\in\NatNum$, $p\ge\PShearerKZ$ and $\mathcal{Z}:=\Set{\mathcal{Z}_n}_{n\in\NatNumZero}$ be $\ShearerMeasure{\FuzzKN}{p}$-distributed (shifting indices by $1$). Define the \emph{canonical model of $k$-independent site percolation with parameter $p$}, abbreviated $\PercolationCanonicalK{p}$, as the $\Tree$-fission of $\mathcal{Z}$.
\end{Mod}

\begin{Rem}
We named our canonical model after the canonical model of \NameBalisterBollobas{} \cite{BalisterBollobas_onepercolation}. Their model is a bond percolation model, whose limit case is defined in the following way: for $p\ge\frac{3}{4}$ let $\TheXi\ge\frac{1}{2}$ be the unique solution of $1-p=\TheXi(1-\TheXi)$ (compare with \eqref{eq:theXi}). Define the bond percolation $Z:=\Set{Z_e}_{e\in E}$ by
\begin{equation}\label{eq:canonicalModelBBDefinition}
	Z_e:=1-(1-X_{\Parent{v}})X_v\,,
\end{equation}
where $e:=(\Parent{v},v)$. See also figure \ref{fig:canonicalModelBB}. Hence it has dependency parameters $k=s=1$. We see that $Y_e$ is closed iff $(X_{\Parent{v}},X_v)=(0,1)$ and comparing it with \eqref{eq:shearerKFuzzZConstruction} we deduce that it is $\ShearerMeasure{\IntNum}{p}$-distributed along downrays. \NameBalisterBollobas{} do not mention this link explicitly, though. They not only use this model in its role as nonpercolating counterexample for a lower bound on $\PMaxE{1}$, as we do with our canonical model in proposition \ref{prop:canonicalModelFacts}, but also show that it has the smallest probability to percolate among all percolations in $\PercolationClassRooted{p}{o}{1}{1}{E}$, their equivalent to our calculations in section \ref{sec:upperBoundPMax}.\\

\NameBalisterBollobas{}' explicit construction is easily generalizable to bond models with higher $k$, but only for $s\ge 2k-1$. Furthermore their inductive approach fails us already for $k\ge 2$. Thus its main inspiration has been to look for $k,0$-independent percolation models being $\ShearerMeasure{\FuzzKZ}{p}$-distributed along all downrays, leading to the tree-fission and our construction in model \ref{mod:canonical}.
\end{Rem}

\begin{figure}[!ht]
\begin{center}
\begin{displaymath}
	\MxyCanonicalModelSupportBB
\end{displaymath}
\end{center}
\caption[\NameBalisterBollobas{}' canonical model]{Construction of \NameBalisterBollobas{}' canonical model. See \eqref{eq:canonicalModelBBDefinition}.}
\label{fig:canonicalModelBB}
\end{figure}

\begin{Prop}\label{prop:canonicalModelFacts}
For all $k\in\NatNum:\PercolationCanonicalK{p}\in\PercolationClassRooted{p}{o}{k}{0}{V}$. If $\BranchingNumberT\le\frac{k+1}{k}$ and $p\in\left[\PShearerKZ,\ExplicitGK(\BranchingNumberT)\right[$, then $\PercolationCanonicalK{p}$ does not percolate. This implies that
\begin{equation}\label{eq:canonicalModelBound}
	\ForAll k\in\NatNum,\, \BranchingNumberT\in\left[1,\frac{k+1}{k}\right[\,:\quad
	\PMaxV{k,0}\ge\ExplicitGK(\BranchingNumberT)\,.
\end{equation}
\end{Prop}

\begin{proof}
As $\mathcal{Z}$ from model \ref{mod:canonical} is $k$-independent and has marginal parameter $p$ corollary \ref{cor:treeFissionInRootedClass} asserts that $\PercolationCanonicalK{p}\in\PercolationClassRooted{p}{o}{k}{0}{V}$.\\

Remember that $p<\ExplicitGK(\BranchingNumberT)$ is equivalent to $\TheXi<\BranchingNumberTReciprocal$, hence we can choose $\varepsilon>0$ such that $(1+\varepsilon)\TheXi<\BranchingNumberTReciprocal$. The first moment method (lemma \ref{lem:firstMomentMethod}) yields
\begin{align*}
	&\FirstAlign \Proba(o\ConnectedTo\infty)\\
	&\le \inf_{\Pi\in\CutsetsOf{o}} \sum_{v\in\Pi}
		\Proba(o\ConnectedTo v)\\
	&\le \inf_{\Pi\in\CutsetsOf{o}} \sum_{v\in\Pi}
		C [(1+\varepsilon)\TheXi]^{\NodeLevel{v}+1}
		&\text{by \eqref{eq:shearerKFuzzZCriticalFunctionMajoration}}\\
	&= C(1+\varepsilon)\TheXi
		\inf_{\Pi\in\CutsetsOf{o}} \sum_{v\in\Pi}
		\left[
			\frac{1}{(1+\varepsilon)\TheXi}
		\right]^{-\NodeLevel{v}}\\
	&=0
		&\text{by definition of $\BranchingNumberT$ in \eqref{eq:branchingNumber}.}
\end{align*}
Therefore $\PercolationCanonicalK{p}$ does not percolate and \eqref{eq:canonicalModelBound} follows directly.
\end{proof}
\subsubsection{The cutup model}\label{sec:cutup}
For $N\in\NatNum$ and $\PShearerKL{N}<\PShearerKZ$ we derive a $k,0$-independent percolation model from $\ShearerMeasure{\FuzzKL{N}}{\PShearerKL{N}}$. It never percolates. In the limit $N\to\infty$ this yields a lower bound of $\PShearerKZ$ for $\PMaxV{k,0}$.

\begin{Mod}\label{mod:cutup}
Let $k,N\in\NatNum$ and $\mathcal{Z}:=\Set{\mathcal{Z}_n}_{n\in\NatNumZero}$ be distributed like independent copies of $\ShearerMeasure{\FuzzKL{N}}{\PShearerKL{N}}$ on $\Set{mN,mN+1,\dotsc,(m+1)N-1}$ for all $m\in\NatNumZero$. Define the \emph{$N$-cutup model of $k$-independent site percolation}, abbreviated $\PercolationCutNK$, as the $\Tree$-fission of $\mathcal{Z}$.
\end{Mod}

\begin{Prop}\label{prop:cutupModelFacts}
For all $k,N\in\NatNum: \PercolationCutNK\in\PercolationClassRooted{\PShearerKL{N}}{o}{k}{0}{V}$. It has percolation cluster diameters \AlmostSureLy{} bounded by $4N-4$. Hence it does not percolate. This implies that
\begin{equation}\label{eq:cutupModelBound}
	\ForAll k\in\NatNum:\quad
	\PMaxV{k,0}\ge\PShearerKZ\,.
\end{equation}
\end{Prop}

\begin{Rem}
It is possible to generate models like the cutup model for every $p<\PShearerKZ$ \cite[proof of theorem 1]{Shearer_problem}.
\end{Rem}

\begin{proof}
As $\mathcal{Z}$ from model \ref{mod:cutup} is $k$-independent and has marginal parameter $p$ corollary \ref{cor:treeFissionInRootedClass} asserts that $\PercolationCutNK\in\PercolationClassRooted{\PShearerKL{N}}{o}{k}{0}{V}$.\\

To bound cluster diameters note that $\ShearerMeasure{\FuzzKL{N}}{\PShearerKL{N}}$ blocks going more than $2N-2$ steps up or down along a downray. Hence cluster diameters are \AlmostSureLy{} bounded by $4N-4$ and $\PercolationCutNK$ does not percolate. Thus $\PMaxV{k,0}\ge\PShearerKL{N}$. Finally we know from \eqref{eq:pShearerKFuzzZ} that $\PShearerKL{N}\xrightarrow[N\to\infty]{}\PShearerKZ$.
\end{proof}
\subsection{Determining \texorpdfstring{$\PMinV{k,s}$}{the minimal p}}
\label{sec:determiningPMin}
To determine $\PMinV{k,s}$ we take the opposite approach from $\PMaxV{k,s}$. For a uniform lower bound we use the first moment method in proposition \ref{prop:firstMomentBound} on percolations with small enough $p$. An upper bound follows from the so-called minimal model \ref{mod:minimal}, again built by tree-fission from section \ref{sec:treeFission}. We show that it percolates for sufficiently high $p$ employing the sufficient conditions on the percolation kernel from section \ref{sec:percolationKernelEstimates}, effectively using the second moment method.

\begin{Prop}\label{prop:firstMomentBound}
\begin{equation}\label{eq:firstMomentBound}
\ForAll k\in\NatNumZero, s\in\NatNumZero\uplus\Set{\infty}:\quad
	\PMinV{k,s}\ge\frac{1}{\BranchingNumberT^{k+1}}\,.
\end{equation}
\end{Prop}

\begin{proof}
Let $\Percolation\in\PercolationClassRooted{p}{o}{k}{s}{V}$ with $p<\frac{1}{\BranchingNumberT^{k+1}}$. Then the first moment method (lemma \ref{lem:firstMomentMethod}) results in
\begin{align*}
	&\FirstAlign \Proba(o\ConnectedTo\infty)\\
	&\le \inf_{\Pi\in\CutsetsOf{o}}
		\sum_{v\in\Pi} \Proba(o\ConnectedTo v)\\
	&\le \inf_{\Pi\in\CutsetsOf{o}}
		\sum_{v\in\Pi} p^{\FracCeiling{\NodeLevel{v}}{k+1}}
		&k\text{-independence along downrays}\\
	&\le \inf_{\Pi\in\CutsetsOf{o}}
		\sum_{v\in\Pi}
		\left(p^{-\frac{1}{k+1}}\right)^{-(k+1)\FracCeiling{\NodeLevel{v}}{k+1}}\\
	&\le \inf_{\Pi\in\CutsetsOf{o}}
		\sum_{v\in\Pi}
		\left(p^{-\frac{1}{k+1}}\right)^{-\NodeLevel{v}}
		&\text{as }(k+1)\FracCeiling{\NodeLevel{v}}{k+1}>\NodeLevel{v}\\
	&=0
		&\text{as }\BranchingNumberT<p^{-\frac{1}{k+1}}\,.
\end{align*}
Hence $\Percolation$ does not percolate and \eqref{eq:firstMomentBound} follows trivially.
\end{proof}

\begin{Mod}\label{mod:minimal}
Let $X:=\Set{X_n}_{n\in\NatNumZero}$ be an \Iid{} \NameBernoulli{} field with parameter $\hat{p}:=p^{1/(k+1)}$. Define $\mathcal{Z}:=\Set{\mathcal{Z}_n}_{n\in\NatNumZero}$ by $\ForAll n\in\NatNumZero: \mathcal{Z}_n := \prod_{i=0}^k X_{n+i}$. Define the \emph{minimal model of $k$-independent site percolation with parameter $p$}, abbreviated $\PercolationMinimalK{p}$, as the $\Tree$-fission of $\mathcal{Z}$.
\end{Mod}

\begin{Prop}\label{prop:minimalModelFacts}
For all $k\in\NatNum: \PercolationMinimalK{p}\in\PercolationClassRooted{p}{o}{k}{0}{V}$. If $p>\frac{1}{\BranchingNumberT^{k+1}}$, then $\PercolationMinimalK{p}$ percolates, which entails that
\begin{equation}\label{eq:minimalModelBound}
	\ForAll k\in\NatNum:\quad
	\PMinV{k,0}\le\frac{1}{\BranchingNumberT^{k+1}}\,.
\end{equation}
\end{Prop}

\begin{proof}
As $\mathcal{Z}$ from model \ref{mod:minimal} is $k$-independent and has marginal parameter $p$ corollary \ref{cor:treeFissionInRootedClass} asserts that $\PercolationMinimalK{p}\in\PercolationClassRooted{\PShearerKL{N}}{o}{k}{0}{V}$.\\

Let $Z:=\Set{Z_v}_{v\in V}$ be $\PercolationMinimalK{p}$-distributed and $p>\frac{1}{\BranchingNumberT^{k+1}}$. Looking at model \ref{mod:minimal}, we see that $\Proba(\mathcal{Z}_{\IntegerSet{n}}=\vec{1})=\Proba(X_{\IntegerSet{n+k}}=\vec{1})=\hat{p}^{n+k}$, with $\hat{p}=p^{1/(k+1)}$. Use the notation from figure \ref{fig:percolationKernel} and apply the bound on the percolation kernel \eqref{eq:ksIndependentKernelBound} to arrive at:
\begin{equation*}
	\PercolationKernel(v,w)
	\le \frac{1}{\Proba(o\ConnectedTo t|t\ConnectedTo w)}
	\le \frac{1}{\hat{p}^{\NodeLevel{t}}}
	\le \hat{p}^{-k-1}\hat{p}^{-\NodeLevel{u}}
\end{equation*}

Apply proposition \ref{prop:secondMomentMethodTreeAdaption} with $C:=\hat{p}^{-k-1}$ and $\alpha:=\frac{1}{\hat{p}}<\BranchingNumberT$ to show that we percolate. This proves \eqref{eq:minimalModelBound}.
\end{proof}

\subsection{The connection with quasi-independence}
\label{sec:connectionWithQuasiIndependence}
In this section we show that in both cases (propositions \ref{prop:secondMomentBound} and \ref{prop:minimalModelFacts}) where we apply the second moment method via exponential bounds on the percolation kernel our $k,s$-independent percolations are also quasi-independent \eqref{eq:quasiIndependent}. This gives an a posteriori connection with \NameLyons{}' work and explains why we have been able to exploit percolation kernels so effectively.

\begin{Prop}\label{prop:ksIndependentImpliesQuasiIndependent}
Let $p>\PShearerKZ$. Then $\ForAll\Percolation\in\PercolationClassRooted{p}{o}{k}{s}{V},\ForAll v,w\in V$:
\begin{equation}\label{eq:ksIndependentQuasiIndependentKernelMajoration}
	\PercolationKernel(v,w)
	\le\frac%
		{\TheXi^{k-(k\lor s)}}
		{(k+1)\TheXi-k}
		\times\frac{1}{\Proba(o\ConnectedTo u)}\,,
\end{equation}
hence $\Percolation$ is quasi-independent.
\end{Prop}

\begin{Rem}(by Temmel)
It is an artefact of our use of \eqref{eq:ksIndependentKernelBound} that we can not show \eqref{eq:ksIndependentQuasiIndependentKernelMajoration} to hold for $p=\PShearerKZ$, where $\TheXi=\frac{k}{k+1}$, and the rhs of \eqref{eq:ksIndependentQuasiIndependentKernelMajoration} explodes. I believe this artefact to be genuine and conjecture that quasi-independence does not hold for $\PercolationCanonicalK{\PShearerKZ}$. This is related to a the relation of \NameShearersMeasure{} with hard-core lattice gases and non-physical singularites of the partition function \cite[section 8]{ScottSokal_repulsive}.
\end{Rem}

\begin{proof}
Let $p>\PShearerKZ$. We use the notation from figure \ref{fig:percolationKernel}.
Then the minimality of \NameShearersMeasure{} \eqref{eq:shearerMinimalityOVOEP}, the explicit minoration on $\FuzzKZ$ in \eqref{eq:shearerKFuzzZOVOEPMinoration} and the fact that $\NodeLevel{t}\le\NodeLevel{u}+(k\lor s)+1$ imply that
\begin{align*}
	&\FirstAlign\Proba(o\ConnectedTo t|t\ConnectedTo w)\\
	&= \Proba(u\ConnectedTo t|o\ConnectedTo u,t\ConnectedTo w)
	   \Proba(o\ConnectedTo u|t\ConnectedTo w)\\
	&= \Proba(u\ConnectedTo t|o\ConnectedTo u,t\ConnectedTo w)
	   \Proba(o\ConnectedTo u)\\
	&\ge\ShearerMeasure{\FuzzKZ}{p}(
			Y_{\Set{\NodeLevel{u}+1,\dotsc,\NodeLevel{t}-1}}=\vec{1}
			|Y_{\Set{0,\dotsc,\NodeLevel{u}}}=\vec{1}
			,Y_{\Set{\NodeLevel{t},\dotsc,\NodeLevel{w}}}=\vec{1}
		)\Proba(o\ConnectedTo u)\\
	&\ge \left[\prod_{i=1}^k \ExplicitFK(i)\right] \ExplicitFK(0)^{(k\lor s)-k}
		\Proba(o\ConnectedTo u)\\
	&= \left[(k+1)\TheXi-k\right]
	  \TheXi^{(k\lor s)-k}\Proba(o\ConnectedTo u)\,.
\end{align*}
Together with the bound on $k,s$-independent percolation kernels \eqref{eq:ksIndependentKernelBound} on $\PercolationKernel(v,w)$ this yields \eqref{eq:ksIndependentQuasiIndependentKernelMajoration} and quasi-independence.
\end{proof}

\begin{Prop}\label{prop:minimalModelIsQuasiIndependent}
The minimal percolation model $\PercolationMinimalK{p}$ is quasi-independent.
\end{Prop}

\begin{proof}
We use the notation from figure \ref{fig:percolationKernel}. The explicit construction in model \ref{mod:minimal} with $\hat{p}=p^{1/(k+1)}$ and the fact that $\NodeLevel{t}\le\NodeLevel{u}+k+1$ imply that
\begin{equation*}
	\Proba(o\ConnectedTo t|t\ConnectedTo w)
	=\hat{p}^{\NodeLevel{t}}
	\ge\hat{p}^{\NodeLevel{u}+k+1}
	=\Proba(o\ConnectedTo u)\,.
\end{equation*}
Together with the bound on $k,s$-independent percolation kernels \eqref{eq:ksIndependentKernelBound} we get quasi-independence
\begin{equation*}
	\PercolationKernel(v,w)
	\le\frac{1}{\Proba(o\ConnectedTo t|t\ConnectedTo w)}
	\le\frac{1}{\Proba(o\ConnectedTo u)}\,.
\end{equation*}
\end{proof}

\subsection{A comment on stochastic domination}
\label{sec:commentOnStochasticDomination}

Recall that a percolation $X$ \emph{stochastically dominates} a percolation $Y$ iff there is a coupling of $X$ and $Y$ such that $\Proba(X\ge Y)=1$. Here the natural order is the partial component-wise order on $\Set{0,1}^E$. We show that for $k\ge 1$ our bounds do not imply stochastic domination of an independent percolation by all $k$-independent percolations for high enough $p$.

\begin{Prop}\label{prop:noBoundForStochasticDomination}
$\ForAll k\ge 1, p\in[0,1[, b\in[1,\infty[: \Exists \hat{p}\in[p,1[$ and $\Tree$ with $\BranchingNumberT=b$ and a $k$-independent site percolation $Z$ on $\Tree$ with parameter $\hat{p}$ such that $Z$ stochastically dominates only the trivial \NameBernoulli{} product field.
\end{Prop}

\begin{Rem}
It is possible to extend proposition \ref{prop:noBoundForStochasticDomination} to all $(\hat{p},b)\in [0,1[\times[1,\infty[$, using \cite[proof of theorem 1]{Shearer_problem}.
\end{Rem}

\begin{proof}
Denote the $d$-regular tree by $\Tree_d$. We know that $\PShearer{\Tree_d}=1-\frac{(d-1)^{(d-1)}}{d^d}$ \cite[theorem 2]{Shearer_problem}. Choose $d$ such that $\PShearer{\Tree_d}>p$. By the definition of $\PShearer{\Tree_d}$ \eqref{eq:pShearerInfinite} there is a finite subtree $\hat{\Tree}$ of $\Tree$ with
\begin{equation*}
	p<\hat{p}:=\PShearer{\hat{\Tree}}<\PShearer{\Tree}\,.
\end{equation*}
Root $\hat{\Tree}$ at some vertex $\hat{o}$. Replace every edge of $\hat{\Tree}$ by a length $(k+1)$ path. Add an extra path of $(k+1)$ edges at $\hat{o}$ with endpoint $\bar{o}$. Extend this finite tree further to some arbitrary infinite tree $\bar{\Tree}$ with branching number $b$ and root it at $\bar{o}$.\\

For every length $(k+1)$ path in the previous paragraph take its last edge and denote their union by $S$. Place $\ShearerMeasure{\hat{\Tree}}{\hat{p}}$ on $S$ and fill up the other edges with \Iid{} Bernoulli($\hat{p}$) variables independently of $\ShearerMeasure{\Tree'}{\hat{p}}$ on $S$. The resulting percolation is $k,0$-independent. By \eqref{eq:pShearerFinite} $\ShearerMeasure{\hat{\Tree}}{\hat{p}}$ fulfils $\ShearerMeasure{\hat{\Tree}}{\hat{p}}(Y_{\Vertices{\hat{\Tree}}}=\vec{1})=0$ and hence the subpercolation on $S$ dominates only the trivial \NameBernoulli{} product field.
\end{proof}
\section*{Acknowledgements}
The second author thanks Wolfgang \NameWoess{} for supporting his journeys and constructive comments on the presentation of this work as well as Yuval \NamePeres{} and Rick \NameDurrett{} for pointing him to \cite{LiggettSchonmannStacey_domination}. This work has been partly effectuated during a series of stays of the second author in Marseille, supported by grants A3-16.M-93/2009-1 and A3-16.M-93/2009-2 from the Bundesland Steiermark and by the Austrian Science Fund (FWF), project W1230-N13.
\bibliography{COMMON/references}

\begin{thebibliography}{10}

\bibitem{AaronsonGilatKeaneDeValk_algebraic}
J.~Aaronson, D.~Gilat, M.~Keane, and V.~de~Valk.
\newblock An algebraic construction of a class of one-dependent processes.
\newblock {\em Ann. Probab.}, 17(1):128--143, 1989.

\bibitem{BalisterBollobas_onepercolation}
N.~P. Balister and B.~Bollob\'{a}s.
\newblock Random geometric graphs and dependent percolation.
\newblock Presentation notes taken by Mathieu, Pierre in Paris during the IHP
  trimester ``Phenomena in High Dimensions'', 2006.

\bibitem{Billingsley_probability}
P.~Billingsley.
\newblock {\em Probability and measure}.
\newblock Wiley Series in Probability and Mathematical Statistics: Probability
  and Mathematical Statistics. John Wiley \& Sons Inc., New York, third
  edition, 1995.

\bibitem{ErdosLovasz_hypergraphs}
P.~Erd{\H{o}}s and L.~Lov{\'a}sz.
\newblock Problems and results on {$3$}-chromatic hypergraphs and some related
  questions.
\newblock In {\em Infinite and finite sets ({C}olloq., {K}eszthely, 1973;
  dedicated to {P}. {E}rd{\H o}s on his 60th birthday), {V}ol. {II}}, pages
  609--627. Colloquia Mathematica Societatis J\'anos Bolyai, Vol. 10.
  North-Holland, Amsterdam, 1975.

\bibitem{FordFulkerson_flows}
L.~R. Ford, Jr. and D.~R. Fulkerson.
\newblock {\em Flows in networks}.
\newblock Princeton University Press, Princeton, N.J., 1962.

\bibitem{LiggettSchonmannStacey_domination}
T.~M. Liggett, R.~H. Schonmann, and A.~M. Stacey.
\newblock Domination by product measures.
\newblock {\em Ann. Probab.}, 25(1):71--95, 1997.

\bibitem{Lyons_rw_percolation_tree}
R.~Lyons.
\newblock Random walks and percolation on trees.
\newblock {\em Ann. Probab.}, 18(3):931--958, 1990.

\bibitem{Lyons_rw_capacity}
R.~Lyons.
\newblock Random walks, capacity and percolation on trees.
\newblock {\em Ann. Probab.}, 20(4):2043--2088, 1992.

\bibitem{LyonsPeres_prob_tree_network}
R.~Lyons and Y.~Peres.
\newblock {\em Probability on trees and networks}.
\newblock Cambridge University Press, 2011.
\newblock in preparation and available online at
  \url{mypage.iu.edu/~rdlyons/prbtree/prbtree.html}.

\bibitem{MathieuTemmel_kindependent}
P.~Mathieu and C.~Temmel.
\newblock K-independent percolation on trees.
\newblock {\em {S}tochastic processes and applications}, 2011.

\bibitem{ScottSokal_repulsive}
A.~D. Scott and A.~D. Sokal.
\newblock The repulsive lattice gas, the independent-set polynomial, and the
  {L}ov\'asz local lemma.
\newblock {\em J. Stat. Phys.}, 118(5-6):1151--1261, 2005.

\bibitem{Shearer_problem}
J.~B. Shearer.
\newblock On a problem of {S}pencer.
\newblock {\em Combinatorica}, 5(3):241--245, 1985.

\bibitem{Temmel_kindependent}
C.~Temmel.
\newblock K-independent percolation on infinite trees after the works of
  {B}ollob\'{a}s and {B}alister.
\newblock Master's thesis, CMI, UdP, Marseille \& Institut für Mathematische
  Strukturtheorie, TUG, Graz, 2008.

\end{thebibliography}
\section{Additional material}
In the proof of proposition \ref{prop:treeFission} the consistency and properties of the family $\Set{\nu_R}_{R\in\mathcal{S}}$ have not been shown.

\begin{proof} Let $R,T\in\mathcal{S}$. Then $\ForAll \vec{s}_{R\cup T}\in\Set{0,1}^{R\cup T}$:
\begin{multline*}
	\nu_{R\cup T}(Y_{R\cup T}=\vec{s}_{R\cup T})
	= \nu_{R\cap T}(Y_{R\cap T}=\vec{s}_{R\cap T})\\
		\times\left(
		\prod_{v\in R\setminus T}
		\Proba(\mathcal{Z}_{\NodeLevel{v}}=s_v|
		      \ForAll w\in A(v):\mathcal{Z}_{\NodeLevel{w}}=s_w)
		\right)
		\left(
		\prod_{v\in T\setminus R}
		\dotsi
		\right)\,.
\end{multline*}
Hence $\nu_S$ and $\nu_T$ coincide on their common support $\Set{0,1}^{R\cap T}$. This implies consistency of the family $\Set{\nu_R}_{R\in\mathcal{S}}$.\\

It remains to show that $\nu_R$ is a probability measure on $\Set{0,1}^R$ with properties \eqref{eq:treeFission}. We prove this by induction over the cardinality of $R$. The induction base for $R=\Set{o}$ is
\begin{equation*}
	\nu_{\Set{o}}(Y_o=0)+\nu_{\Set{o}}(Y_o=1)
	=\Proba(\mathcal{Z}_0=0)+\Proba(\mathcal{Z}_0=1)
	= 1\,.
\end{equation*}
The induction step reduces $R$ to $T:=R\setminus\Set{v}$ for some leaf $v$ of $\Graph{R}$. Hence
\begin{align*}
	&\FirstAlign \sum_{\vec{s}_R} \nu_R(Y_R=\vec{s}_R)\\
	&= \sum_{\vec{s}_T} \sum_{s_v\in\Set{0,1}}
			\nu_R(Y_v=s_v,Y_{T}=\vec{s}_T)\\
	&= \sum_{\vec{s}_T}
			\nu_R(Y_T=\vec{s}_T)
		\underbrace{\sum_{s_v\in\Set{0,1}}
			\Proba(\mathcal{Z}_{\NodeLevel{v}}=s_v|
			      \ForAll w\in A(v): \mathcal{Z}_{\NodeLevel{w}}=\vec{s}_w)
		}_{=1}\\
	&= \sum_{\vec{s}_T}
			\nu_T(Y_T=\vec{s}_T)\\
	&= 1\,.
\end{align*}
For independence suppose that $W\subseteq R\in S$ fulfills the condition of \eqref{eq:treeFissionSubtreeIndependence}. Let $U:=\bigcup_{w\in W} A(w)$ and for $w\in W$ let $V_w:=\Vertices{\SubtreeRootedAt{\Tree}{w}}\cap R$. Then \eqref{eq:treeFissionDefinition} entails that
\begin{equation*}
	\nu_R(\ForAll w\in W: Z_{V_w}=\vec{s}_{V_w}|Z_U=\vec{s}_U)
	=\prod_{w\in R} \nu_R(Z_{V_w}=\vec{s}_{V_w}|Z_U=\vec{s}_U)\,.
\end{equation*}
Conditional independence on $Z_U$ implies independence as in \eqref{eq:treeFissionSubtreeIndependence}.\\

We turn to the distribution along downpaths. For $v\in V$ we have $P:=\TreePath{o}{v}=:\Set{o=:w_0,\dotsc,w_{\NodeLevel{v}}:=v}\in S$. Hence $\ForAll \vec{s}_P\in\Set{0,1}^P$:
\begin{align*}
	&\FirstAlign\nu_P(Z_P=\vec{s}_P)\\
	&=\prod_{i=0}^{\NodeLevel{v}}
		\nu_P(Z_{w_i}=s_{w_i}|Z_{A(w_i)}=\vec{s}_{A(w_i)})\\
	&=\prod_{i=0}^{\NodeLevel{v}}
		\Proba(\mathcal{Z}_i=s_{w_i}|
		      \ForAll w\in A(w_i):\mathcal{Z}_{\NodeLevel{w}}=s_w)\\
	&=\Proba(\ForAll w\in P: \mathcal{Z}_{\NodeLevel{w}}=s_w)\,.
\end{align*}

Finally the invariance under automorphisms of the rooted tree is a result of the obliviousness of the construction to the ordering of the children.
\end{proof}

\end{document}